\documentclass[twoside,a4paper,11pt]{amsart}

\usepackage{amsmath}
\usepackage{amssymb}
\usepackage{comment}
\usepackage{graphicx}

\usepackage{amsxtra}
\usepackage{mathrsfs}

\newtheorem{theorem}{Theorem}
\newtheorem{lemma}{Lemma}[section]
\newtheorem{proposition}[lemma]{Proposition}

\theoremstyle{definition}

\numberwithin{equation}{section}

\newcommand{\beq}{\begin{equation}}
\newcommand{\eeq}{\end{equation}}
\newcommand{\RE}{\mathbb R}

\newcommand{\R}{\mathbb R}
\newcommand{\CO}{\mathbb C}

\newcommand{\N}{\mathbb N}

\newcommand{\E}{\mathbb E}
\newcommand{\PP}{\mathbb P}

\newcommand{\Tr}{\operatorname{Tr}}

\newcommand{\ve}{\varepsilon}
\newcommand{\al}{\alpha}

\newcommand{\T}{\mathbb{T}}

\renewcommand{\Im}{\operatorname{Im}\,}
\renewcommand{\Re}{\operatorname{Re}\,}

\newcommand{\xx}{{\bf x}}

\newcommand{\J}{\mathbb J}

\title[]{Local Marchenko-Pastur law at the hard edge of the sample covariance ensemble}

\author[]{Anastasis Kafetzopoulos and Anna Maltsev}

\date{\today}

\begin{document}

\maketitle

\begin{abstract}
Consider an $N$ by $N$ matrix $X$ of complex entries with iid real and imaginary parts. We show that the local density of eigenvalues of $X^*X$ converges to the Marchenko-Pastur law on the optimal scale with probability 1. We also obtain rigidity of the eigenvalues in the bulk and near both hard and soft edges. Here we avoid logarithmic and polynomial corrections by working directly with high powers of expectation of the Stieltjes transforms. We work under the assumption that the entries have a finite 4th moment and are truncated at $N^{1/4}$. In this work we simplify and adapt the methods from prior papers of G\"otze-Tikhomirov and Cacciapuoti-Maltsev-Schlein to covariance matrices.
\end{abstract}

\section{Introduction}
In this paper we obtain optimal large deviation bounds on the Stieltjes transform for the sample covariance random matrix ensemble. Let $X$ be a $ M\times N$ matrix with components $x_{ij} = \Re x_{ij} + i \Im x_{ij}$. Assume that $\Re x_{ij}$ and $\Im x_{ij}$ are independent identically distributed (iid) real random variables with mean zero and variance $\frac12$ so that
\beq
\E x_{ij} =0\quad\textrm{and}\quad \E|x_{ij}|^2 =1 \qquad i=1,...,N, \, j = 1, \dots, M\,.
\eeq
and
\[
d :=M/N.
\]
In what follows we shall denote by $X_N$ the scaled matrix
\beq
\label{e:XN}
X_N=X/\sqrt{N}.
\eeq
 We are interested in the analysis of the asymptotic empirical spectral measure of the matrix $X_N^* X_N$ for $N\to \infty$, when $M = N$. This is the case when the limiting measure has a square root singularity near 0 with typical distance between eigenvalues on the order of $\frac{1}{N^2}.$ We are able to obtain results on the hard edge, the bulk, and the soft edge in a unified way.

\medskip
Let $s_\al$, $\al=1,...,N$, be the eigenvalues of $X_N^*X_N$. Since $X_N^*X_N$ is Hermitian and positive definite we can assume that $0\leq s_1\leq s_2\leq . . .  \leq s_N$. We denote by $n_N$  the empirical spectral distribution of the eigenvalues $s_\al$,
\beq
n_N(E)=\frac1N\#\{\al\leq N\,|\; s_\al \leq E\}\,
\eeq
and
\begin{equation}
\mathcal{N}(I) = \#\{\al\leq N\,|\; s_\al \in I\}
\end{equation}

For any $\theta \in\CO$ with $\Im \theta \neq 0$ we define the  Stieltjes transform of $n_N$ as
\beq
\label{e:DeN}
\Delta_N(\theta)
=
\int_\RE \frac{1}{x-\theta} dn_N(x)=
\frac{1}{N} \Tr (X_N^*X_N -\theta )^{-1}
=
\frac{1}{N}\sum_{\alpha=1}^N \frac{1}{s_\al - \theta}\,.
\eeq

We denote by $\nu$ the probability distribution of $\Re x_{ij}$ and $\Im x_{ij}$. In this paper we assume that
\begin{equation} \label{a:4moment}\sup_{N\geq 1} \sup_{1 \leq j,k \leq N} \E|x_{jk}|^4 =: \mu_4 < \infty  ,\end{equation} %
and that there exists a constant $D>0$ such that for all $N$:
\begin{equation}\label{a:truncation} \sup_{1 \leq j,k \leq N} |x_{jk}| \leq D N^{1/4} . \end{equation} These assumptions are the same as in the papers of G\"otze-Tikhomirov \cite{gotze2014rate, gotze2016optimal}, and with easy modifications all the proofs and results hold as well for $x_{ij}$ such that $\E |x_{ij}|^q \leq (Cq)^{cq}$ for universal constants $C, c$.

\medskip

The first results about universality of covariance matrices date back to '67. Let \[\lambda_{\pm} = (1 \pm \sqrt d)^2.\]
 Marchenko Pastur in \cite{marchenko1967distribution} show that $d\nu_N \rightarrow \rho$ weakly with probability 1, where $\rho$ is the Marchenko-Pastur distribution, given by
\begin{equation}\label{MP}
\rho_{MP}(E) = \frac{1}{2\pi} \sqrt{\frac{(\lambda_{+}-E)(E - \lambda_{-})}{E^2}},
\end{equation}
whenever $E \in [\lambda_{-}, \lambda_{+}]$ and 0 otherwise.
In the case of a square matrix $X$, the density of the Marchenko-Pastur distribution is
\beq
\label{e:rhoMP}
\rho(E) =
\left\{
\begin{aligned}
&\frac{1}{2\pi} \sqrt{\frac{4}{E}-1} & \qquad 0<E\leq 4 \\
&0&\textrm{otherwise}
\end{aligned}
\right.
\eeq
and for  any $\theta$ such that $\Im \theta\neq 0$ we denote by $\Delta$ the associated Stieltjes transform
\beq
\label{e:De}
\Delta(\theta )= \int_{\RE} \frac{1}{x-\theta} \rho (x) dx\,
\eeq
which satisfies the quadratic equation
\begin{equation}
\Delta = -\frac{1}{\theta (\Delta+ 1)}.
\end{equation}
\medskip
In \cite{marchenko1967distribution}, the convergence of the density of states is on intervals whose sizes are independent of $N$. In this case, the intervals that are away from the endpoints contain an order of $N$ eigenvalues. A natural question to study is whether the convergence remains on intervals whose size (we call the interval size scale) goes to zero as $N$ grows.

\medskip
In \cite{erdos2012local}, Erd\"os-Schlein-Yau-Yin establish convergence of the empirical spectral density for general covariance matrices to the Marchenko-Pastur law in the bulk for $d < 1$ on small intervals. They use a decomposition by minors for the diagonal elements of the resolvent to establish a self-consistent equation for the Stieltjes transform $\Delta_N$ of $d\nu_N$. Large deviation estimates and a continuity argument are then used show the convergence of the spectral measure on small intervals (involving polynomial corrections) in the bulk  distribution. These methods have been extended to the hard edge and logarithmic rather than polynomial corrections by Cacciapuoti-Maltsev-Schlein in \cite{cacciapuoti2013local}. More precisely, the authors show that the fluctuation of the Stieltjes transform $\sqrt E\Delta_N$ away from $\sqrt E\Delta$ is on the order of $\sqrt{\frac{\sqrt{E}}{N\eta}}$ and they obtain convergence of the counting function of eigenvalues everywhere including close to the hard edge. Eigenvalue rigidity with polynomial corrections for the bulk and soft edges for entries with subexponential decay can be found in Pillai-Yin \cite{pillai2014universality}.

\medskip
A related question is that of the universality of the correlation function of the eigenvalues. Results in the bulk using local laws and a local relaxation flow can be found in \cite{erdos2012local, pillai2014universality}. A similar result in \cite{tao2012random} by proving a version of the four moment theorem for random covariance matrices for any $0< d \leq 1$ in the bulk of the spectrum. Wang \cite{wang2012random} extends these results to the soft edge (cf Remark 1.8 in \cite{wang2012random}). For the hard edge, universality of the joint distribution of low-lying eigenvalues has been established by Tao-Vu in \cite{tao2010random}. Another related question is about the rate of convergence of the density of states to the Marchenko-Pastur law. In \cite{gotze2014rate}, the authors establish that the Kolmogorov distance between the expected spectral measure and the Marchenko-Pastur law is $O(N^{-1})$. Additionally, there has been some remarkable progress on similar questions in the case of Wigner (matrices with i.i.d. entries up to Hermitian symmetry) and more general Wigner-type matrices \cite{landon2020applications, landon2021single}. The authors use homogenization theory, which relies on coupling two Dyson Brownian motions, to establish the Gaussianity of fluctuation of individual eigenvalues in the bulk of the spectrum.

\medskip

In this paper we obtain optimal bounds on the expectations of high moments of the fluctuation $\Lambda = \Delta_N - \Delta$ on the optimal scale. Our methods and results apply to the bulk as well as the soft and hard edges. The main objective of this work is to extend the results and methodology of \cite{cacciapuoti2015bounds} to a hard edge setup. We were able to simplify the proof of Theorem \ref{t:sti} in \cite{cacciapuoti2015bounds} avoiding different cases for the bulk and edges. Unlike in the Wigner case, where both edges are soft, the presence of the hard edge at 0 allows us to extend the bounds on the real part of the Stieltjes transform to the negative real line, thus also yielding a fluctuation for the individual eigenvalue near the hard edge that is decreasing with the eigenvalue number. This paper also improves on \cite{cacciapuoti2013local} by removing the logarithmic corrections and improving the fluctuation bounds. We also extended the proofs in \cite{gotze2016optimal,gotze2014rate} on fluctuations of quadratic forms to a soft edge setup by improving a factor of $|\Delta|$ to a factor of $\Im \Delta$.

\medskip
To state our theorem we define the domain $S_{E, \eta}$ where we obtain our bounds:
\begin{equation}\label{e:setae}
S_{E, \eta}:= \{4\eta > c(E^2 + \eta^2 -4E) \}
\end{equation}
for some $c >0$. This domain is chosen so that $\Im (\Delta + 1/2)^2 \geq c\Re ((\Delta + 1/2)^2)$ which we need for the proof of Proposition \ref{b-R}. While all the proofs work for all $c>0$ not dependent on $N$, we will specifically work with $c=1$ to allow us the opportunity to illustrate it the following picture, Figure \ref{f:domains}.
\begin{figure}[h]\label{f:domains}
\centering{\includegraphics[scale=0.5]{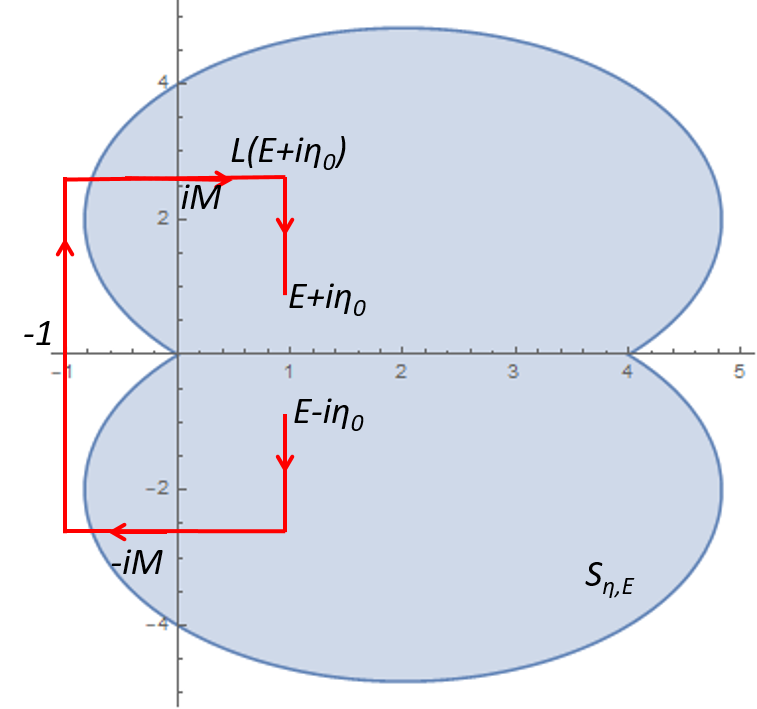}
\caption{The set $S_{\eta, E}$ shaded in blue and the integration contour $L(z_0)$ from \eqref{e:pleijel} in red}
}
\end{figure}

\begin{theorem}
\label{t:sti}
Let $X_N$ be a $N\times N $ matrix as described in equation \eqref{e:XN}, and assume \eqref{a:4moment} and \eqref{a:truncation}. Let $\Delta_N$ and $\Delta$ be the Stieltjes transforms defined in equations \eqref{e:DeN} and \eqref{e:De}. Moreover set $\theta = E + i \eta$, with $\frac{N\eta}{|\sqrt{\theta}|} \geq M$  for some suitably large $M$. Then there exist positive constants $c_0$, $C$ such that for each $K>0$ and $ 1\leq q \leq c_0\left ( \frac{N\eta}{|\sqrt{\theta}|} \right )^{1/8} $ and $\theta \in S_{E,\eta}$ or $E < 0$
\beq
 \PP \left( |\sqrt{\theta}| | \Delta_N(\theta)-\Delta(\theta) | \geq K \frac{|\sqrt{\theta}|}{N\eta} \right) \leq \frac{(Cq)^{cq^2}}{K^q} \\
\eeq
Furthermore, for any $E \in \R$ and $\eta>0$ such that $\frac{N\eta}{|\sqrt{\theta}|} \geq M$ we have that
\beq
 \PP \left( |\sqrt{\theta}| |\Im \left(\Delta_N(\theta)-\Delta(\theta)\right) | \geq K \frac{|\sqrt{\theta}|}{N\eta} \right) \leq \frac{(Cq)^{cq^2}}{K^q}.
\eeq
\end{theorem}

We then use our Theorem \ref{t:sti} to obtain fluctuation estimates on the counting function as stated in the next theorem. Letting
\begin{equation}
n_{MP}(E) = \int_0^E \rho(x)dx,
\end{equation}
we compare it to $n_N$.

\begin{theorem}\label{t:counting}
With assumptions as in Theorem \ref{t:sti}, there exist constants $M_0, N_0, C, c >0$ such that for any $K > 0$ and $E \geq \frac{M_0}{N^2}$
\begin{equation}\label{e:counting}
\PP\left(|n_N(E) - n_{MP}(E)| \geq K \min \left\{\sqrt {E}, \frac{ \log N}{N}\right\}\right) \leq \frac{(Cq)^{cq^2}}{K^q}
\end{equation}
for all $E\in \R$, $K>0, N>N_0, q\in \N$.

\end{theorem}

\medskip

We use the above estimate to obtain rigidity estimates that is how far each eigenvalue can fluctuate away from its classical location.
We define the classical locations of the eigenvalues, predicted by the Marchenko-Pastur distribution, as the points $\gamma_a$, ($a=1,...,N$) such that
$$ \int_0^{\gamma_a} \rho(E) dE = \frac{a}{N}. $$
In particular, we obtain the fluctuation of eigenvalues near the hard edge to be of the order of $\frac{\log N}{N^2}$. The fluctuations of eigenvalues in the bulk and soft edges of both the Gaussian Unitary Ensemble and the Wishart Ensemble are known to be respectively of the order $\frac{\sqrt{\log N}}{N}$ in the bulk and $\frac{\sqrt{\log k}}{k^{1/3}N^{2/3}}$ for the $k$th eigenvalue from the edge, $k \rightarrow \infty$ (see \cite{gustavsson2005gaussian, su2006gaussian}). To our knowledge similar results are not yet available for the hard edge.

\begin{theorem}\label{t:rigidity}
With assumptions as in Theorem \ref{t:sti}, there exist constants $C,c,N_0,\epsilon>0$ such that
\begin{equation} \label{e:rigiditybulk}\mathbb{P} \left ( | \lambda_a  - \gamma_a | \geq K \frac{ \log N}{N} \left ( \frac{a}{N}  \right )\right) \leq \frac{(Cq)^{cq^2}}{K^q} \end{equation}
for $a=1,...,\lceil N/2 \rceil$, $N>N_0$, $K>0$, and $q \in \mathbb{N}$ with $q\leq N^\epsilon.$
Furthermore, for $a \leq \log N$ we have that
\begin{equation} \label{e:rigidityedge}\mathbb{P} \left ( | \lambda_a  - \gamma_a | \geq K \frac{ a^2}{N^2}\right) \leq \frac{(Cq)^{cq^2}}{K^{q/2}}. \end{equation}

\end{theorem}
In this theorem the factor $\frac{a}{N}$ accounts for the higher density at the hard edge. Here we focus on hard-edge rigidity, since proofs of soft-edge rigidity require control of the largest eigenvalue which, to our knowledge, is not currently available in the case of truncated entries with four moments, in either Wigner or Sample Covariance case.

\section{Useful Identities}
In this section we collect some useful known identities. Let $\J, \J_1, \J_2 \subset \{1, ..., N\}$. We will denote by $X^{(\J)}$ the submatrix of $X_N$ with columns of indices $\J$ removed, and $X_{(\J)}$ with rows of indices $\J$ removed.

We define the resolvent matrices
\begin{equation}G^{(\J_1)}_{(\J_2)}:=\left((X^{(\J_1)}_{(\J_2)})^* X^{(\J_1)}_{(\J_2)}-\theta\right)^{-1} \quad \text{and} \quad \mathcal{G}^{(\J_1)}_{(\J_2)}:=\left(X^{(\J_1)}_{(\J_2)}(X^{(\J_1)}_{(\J_2)})^*  - \theta\right)^{-1}.
\end{equation}
 When our arguments work for any $\J_1, \J_2$ we will mention this and then suppress them for ease of notation, and we will write $G^{(\J_1)}_{(\J_2), ij}$ for the $ij$th element. We notice here that $G^{(\J)}$ is the minor of $G: = G^{(\emptyset)}$ with $\J$-th rows and $\J$-th columns removed. Lastly we notice that
\begin{equation}\label{e:traces}\Tr \mathcal{G}^{(\J_1)}_{(\J_2)} =  \frac{|\J_1|- |\J_2|}{\theta} + \Tr G^{(\J_1)}_{(\J_2)}
\end{equation}
Similarly we introduce
\begin{equation}\label{d:LambdaDelta}
\Delta^{(\J_1)}_{N, (\J_2)}: = \Tr G^{(\J_1)}_{(\J_2)} \quad \text{and} \quad \Lambda^{(\J_1)}_{(\J_2)} := \Delta^{(\J_1)}_{N, (\J_2)} - \Delta
\end{equation}
and we use $\Delta_N$ and $\Lambda$ when $\J_1, \J_2 = \emptyset$. We will use $\xx_k$ and $\xx^k$ for rows and columns of $\sqrt{N}X_N$ respectively.

We state some well-known identities for resolvent entries (Lemma 2.3 of \cite{pillai2014universality}).
\begin{lemma} \label{resol}
With $G^{(\J_1)}_{(\J_2)}$ as before for $i, j \neq k$, we have
\begin{equation}\label{e:Gij}
 G^{(\J_1)}_{(\J_2), ij} = G^{(\J_1\cup \{k\})}_{(\J_2), ij} + \frac{G^{(\J_1)}_{(\J_2), ik}G^{(\J_1)}_{(\J_2), ki}}{G^{(\J_1)}_{(\J_2), kk}} .
 \end{equation}

 \end{lemma}

Furthermore, as seen for example in (3.2) of \cite{cacciapuoti2015bounds}, we have the following relationship between the $(k,k)$ element of $G^2$ and $\Im G_{kk}$, and the same holds for $G^{(\J_1)}_{(\J_2)}, \mathcal{G}^{(\J_1)}_{(\J_2)}$:
\begin{equation}\label{e:GsImG}
|G(z)|_{kk}^2 = \frac{(\Im G(z))_{kk}}{\eta}
\end{equation}
yielding that
\begin{equation}\label{e:GsImG2}
|(G^2)_{kk}|\leq \frac{\Im G_{kk}}{\eta}.
\end{equation}

Next we observe that using the proof of (3.10) in \cite{cacciapuoti2015bounds} we can also obtain the following for the resolvent of the sample covariance ensemble, and the proof works for $G_{(\J_2)}^{(\J_1)}, \mathcal{G}_{(\J_2)}^{(\J_1)}$ for any $\J_1, \J_2$:
\begin{lemma}\label{G11eta/s}
With $G$ and $\mathcal{G}$ as before, we have that
\begin{equation}
G_{11}(E+i\eta/s) \leq sG_{11}(E+i\eta).
\end{equation}
\end{lemma}

Furthermore we have the following bounds on the Stieltjes transform of the Marchenko-Pastur law.
For $E>0$ we set $\kappa := |E-4|.$ For any fixed $E_0',E_0>0$ and $\eta_0>0$ there exist constants $C>0$ such that

\beq \label{7} \left | \Delta + \frac{1}{2} \right | \geq C ( \kappa^2   +\eta^2)^{\frac{1}{4}} \geq C\sqrt{\kappa+\eta}, \eeq

and

\beq \label{8}
c \frac{\eta}{\sqrt{\kappa+\eta}} \leq \mathrm{Im} \Delta \leq C \frac{\eta}{\sqrt{\kappa+\eta}} , \eeq

$\forall E_0\leq E \leq E_0', \ 0<\eta\leq \eta_0,$ $\kappa \geq \eta .$

\section{Equations for $\Lambda$}

\begin{lemma}
\label{l:eqDelta}
Take $\theta = E+i\eta$.  For any $N\geq N_0$ one has
\beq
\label{e:eqDeltaN}
\Delta^{(\J_1)}_{N, (\J_2)}
= - \frac{1}{N} \sum_{k=1}^N
\frac{1}{\theta\left(1+\Delta^{(\J_1)}_{N, (\J_2)} + T_k + \Upsilon_{(\J_2)}^{(\J_1\cup \{k\})} \right)
}
=
- \frac{1}{N} \sum_{k=1}^N
\frac{1}{\theta\left(1+ \Delta^{(\J_1)}_{N, (\J_2)} + \mathcal{T}_k + Y_{(\J_2\cup \{k\})}^{(\J_1)} \right)
}
\eeq
with
\beq
\label{e:boundsT12}
|T_k|, |\mathcal{T}_k| \leq \frac{|\,|\J_1| - |\J_2|\,|+1}{N\eta}
\eeq
and
\begin{equation}\label{d:Upsilon}\Upsilon_{(\J_2)}^{(\J_1\cup \{k\})} := (\mathbb{I} - \E_{\textbf{x}^k}) (\textbf{x}^k/\sqrt N)^* \mathcal{G}_{(\J_2)}^{(\J_1\cup \{k\})} \textbf{x}^k/\sqrt N \quad \text{and}\quad Y_{(\J_2\cup \{k\})}^{(\J_1)}: = (\mathbb{I} - \E_{\textbf{x}_k}) \frac{\textbf{x}_k}{\sqrt N} G_{(\J_2\cup \{k\})}^{(\J_1)} \textbf{x}^*_k/\sqrt N.
\end{equation}
\begin{equation}\label{d:Tau}T_k := \frac{1}{N} \Tr \mathcal{G}_{(\J_2)}^{(\J_1) \cup \{k \}} - \frac{1}{N}\Tr G_{(\J_2)}^{(\J_1)} \quad \text{and}\quad \mathcal{T}_k: = \frac{1}{N} \Tr {G}_{(\J_2)}^{(\J_1) \cup \{k \}} - \frac{1}{N}\Tr \mathcal{G}_{(\J_2)}^{(\J_1)},
\end{equation}
where $\textbf{x}^k$ is the $k-$th column of the matrix $X$ and $\textbf{x}_k$ is the $k-$th row of $X$.

\end{lemma}

\begin{proof}

By the definition of $\Delta_N$ and from the formula

\begin{multline}\label{e:Gkk}
G_{(\J_2), kk}^{(\J_1)} = \frac{1}{\frac{|\xx^k|^2}{N} -\theta - (\xx^k)^* X_{(\J_2)}^{(\J_1\cup \{k\})}\left((X_{(\J_2)}^{(\J_1\cup \{k\})})^*X_{(\J_2)}^{(\J_1)\cup \{k\}}-\theta\right)^{-1}(X_{(\J_2)}^{(\J_1\cup \{k\})})^*\xx^k}
\\
= -\frac{1}{\theta \left(1+  (\xx^k/\sqrt N)^* \mathcal{G}_{(\J_2)}^{(\J_1\cup \{k\})}
\xx^k/\sqrt N \right)} = -\frac{1}{\theta\left(1+\Delta_{N, (\J_2)}^{(\J_1)} + T_k + \Upsilon_{(\J_2)}^{(\J_1\cup \{k\})} \right)}
\end{multline}
we obtain that
\beq
\Delta_{N, (\J_2)}^{(\J_1)}
= - \frac{1}{N} \sum_{k=1}^N
\frac{1}{\theta\left(1+ \Delta_{N, (\J_2)}^{(\J_1)} + T_k + \Upsilon_{(\J_2)}^{(\J_1\cup \{k\})} \right)
}
\eeq
where $\Upsilon_{(\J_2)}^{(\J_1\cup \{k\})}$ is as in \eqref{d:Upsilon}
and
\begin{equation}\label{d:Tk}\begin{split}
T_k &= \frac{1}{N} \Tr(X_{(\J_2)}^{(\J_1\cup \{k\})}(X_{(\J_2)}^{(\J_1\cup \{k\})})^* - \theta)^{-1} - \frac{1}{N}\Tr((X_{(\J_2)}^{(\J_1)})^*(X_{(\J_2)}^{(\J_1)}) - \theta)^{-1}
\\
  &= -\frac{|\J_1|-|\J_2|}{N\theta} + \frac{1}{N} \Tr((X_{(\J_2)}^{(\J_1\cup \{k\})})^*X_{(\J_2)}^{(\J_1\cup \{k\})} - \theta)^{-1} - \frac{1}{N}\Tr((X_{(\J_2)}^{(\J_1)})^*X_{(\J_2)}^{(\J_1)} - \theta)^{-1}.
\end{split}\end{equation}
Rewriting, we obtain
\begin{equation}\label{e:minorcalc}\begin{split}
&T_k
= -\frac{|\J_1|-|\J_2|}{N\theta} +\frac{1}{N} \left ( \sum \limits_{i \neq k} G_{ii}^{(k)} - \sum \limits_{i=1}^N G_{ii} \right )
= -\frac{|\J_1|-|\J_2|}{N\theta} +\frac{1}{N} \left( \sum \limits_{i \neq k} \left ( G_{ii} - \frac{G_{ik}G_{ki}}{G_{kk}} \right) - \sum \limits_{i=1}^N G_{ii} \right )
\\
&= -\frac{|\J_1|-|\J_2|}{N\theta} -\frac{1}{N} \left( \sum \limits_{i\neq k} \frac{G_{ik}G_{ki}}{G_{kk}} -G_{kk} \right)
= -\frac{|\J_1|-|\J_2|}{N\theta} -\frac{1}{N} \frac{1}{G_{kk}}\sum \limits_{i=1}^N G_{ik}G_{ki} = -\frac{|\J_1|-|\J_2|}{N\theta} -\frac{(G^2)_{kk}}{NG_{kk}}.
\end{split}\end{equation}
We now use \eqref{e:GsImG} to obtain
\begin{equation}\label{e:minor}
|T_k| \leq \frac{|\,|\J_1|-|\J_2|\,|}{N|\theta|} + \frac{\Im G_{kk}}{|G_{kk}| N \eta}
\end{equation}
yielding that
\[
|\sqrt \theta| \, |T_k| \leq \frac{(|\,|\J_1|-|\J_2|\,|+1)|\sqrt \theta|}{N \eta}.
\]

Note also \begin{equation}\label{e:curlyGkk}
\mathcal{G}_{(\J_2), kk}^{(\J_1)}
= -\frac{1}{\theta \left(1+  (\textbf{x}_k/\sqrt{N}) \left((X_{(\J_2 \cup \{k\})}^{(\J_1)})^*X_{(\J_2 \cup \{k\})}^{(\J_1)}-\theta\right)^{-1}
\textbf{x}_k^*/\sqrt{N} \right)}
\end{equation}
which similarly yields the second part of \eqref{e:eqDeltaN}, recalling \eqref{e:traces}.
\end{proof}

Rewriting \eqref{e:eqDeltaN} using $\frac{1}{A + \epsilon} = \frac{1}{A} - \frac{\epsilon}{A (A+\epsilon)}$
 we obtain (also for any $\J_1, \J_2$, thus we suppress them here)
\beq\begin{split}
\Delta_N
&
= - \frac{1}{N} \sum_{k=1}^N
\frac{1}{\theta(1+ \Delta) + \theta \Lambda  + {\theta}(T_k + \Upsilon^{(\{k\})})
}
\\
&
=-  \frac{1}{\theta (1+\Delta)} - \frac{1}{N} \sum_{k=1}^N
\frac{1}{\theta (1+\Delta)}\,\frac{ \theta \Lambda  + {\theta}(T_k + \Upsilon^{(\{k\})})}{\theta\left(1+\Delta_N + (T_k + \Upsilon^{(\{k\})})\right)}
\\
&
= \Delta - \frac{\Delta}{N} \sum_{k=1}^N\theta \Lambda G_{kk}  + - \frac{\Delta}{N} \sum_{k=1}^N G_{kk} {\theta}(T_k + \Upsilon^{(\{k\})})
\\
&
= \Delta - \Delta \theta \Lambda \Delta_N  + - \frac{\Delta}{N} \sum_{k=1}^N G_{kk} {\theta}(T_k + \Upsilon^{(\{k\})})
\end{split}\end{equation}
This yields that
\begin{equation}
\Lambda =
 \theta \Delta\Lambda (\Delta + \Lambda)  +  \frac{\Delta}{N} \sum_{k=1}^N G_{kk}{\theta}(T_k + \Upsilon^{(\{k\})})
\end{equation}
Let
\begin{align}
R &:= N^{-1} \sum_{k=1}^N G_{kk} (T_k + \Upsilon^{(\{k\})})
\end{align}
and (similarly can define $R^{(\J_1)}_{(\J_2)}$) which yields the following quadratic for $\Lambda$
\beq \label{e:quadlambda}
\theta \Delta \Lambda^2 + (\theta \Delta^2 - 1)\Lambda + \Delta \theta R = 0.
\eeq
Dividing by $\theta \Delta$, using that $\theta\Delta = -\frac{1}{1+\Delta}$ and the quadratic formula, yields
\beq
 -(\Delta + 1/2) \pm \sqrt{(\Delta + 1/2)^2 - R}
\eeq
as two solutions. From definition of $\Lambda$ in \eqref{d:LambdaDelta} it follows that $\Im \Lambda > \Im \Delta$, thus if we take the branch cut of the square root to be on the positive reals so that the imaginary part of the square root is always positive, we obtain that
\begin{equation}
\Lambda =  -(\Delta + 1/2) + \sqrt{(\Delta + 1/2)^2 - R}
\end{equation}
We also notice that the second solution, call it $\tilde{\Lambda}$, to  \eqref{e:quadlambda} is given by
\beq\label{d:tildeLambda}
\tilde{\Lambda} = -\Lambda - 2\Delta - 1.
\eeq
\medskip

The following proposition is analogous to Proposition 2.2 of \cite{cacciapuoti2015bounds}.
\begin{proposition} \label{b-R}
Let $\theta=E+i\eta$. There exists a constant $C>0$, such that:

\begin{equation} \label{e:Lambdabound}
|\Lambda| \leq C \min \left \{ \frac{|R|}{|\Delta + \frac{1}{2}|} , \sqrt{|R|} \right \} ,
\end{equation}

for all $(E, \eta) \in S_{\eta, E} $ as well as for any $E < 0$. Furthermore, for any $E\in \R$ and $\eta > 0$ we have that
\begin{equation}\label{e:imLambdabound}
|\Im \Lambda| < C\min \left\{ \frac{|R|}{|\Delta + \frac{1}{2}|} , \sqrt{|R|} \right \}
\end{equation}
and
\begin{equation}\label{e:twosolutionsbound}
\min\{|\Lambda|, |\tilde{\Lambda}|\} \leq C\sqrt{|R|}.
\end{equation}
Analogous statements hold for $\Lambda_{(\J_2)}^{(\J_1)}$ with $R_{(\J_2)}^{(\J_1)}$.
\end{proposition}

\begin{proof}
To show \eqref{e:Lambdabound}, we apply (2.17) of \cite{cacciapuoti2015bounds} with $a= (\Delta + \frac{1}{2})^2 $ and b = $ - R $. Since $\Im \Delta >0$, with our choice of branch cut we have that $\sqrt{(\Delta+1/2)^2} = \Delta+1/2$, and we recall that we defined $S_{\eta, E}$ in \eqref{e:setae} to be exactly the set where $\Im (\Delta + 1/2)^2 \geq c\Re ((\Delta + 1/2)^2)$ for some $c>0$. Note that \eqref{e:imLambdabound} follows directly from (2.18) of \cite{cacciapuoti2015bounds}, while the proof of \eqref{e:twosolutionsbound} is identical to the proof of (2.16) in \cite{cacciapuoti2015bounds}.

\medskip
Recalling that $\Delta_N(z) = \frac{1}{N}\sum_\alpha \frac{1}{s_\alpha - E - i\eta}$ and noting that for $E < 0$ the real part of each summand is positive we conclude that $\Re \Delta_N > 0$ for $E < 0$, and similar to our argument about the imaginary part of $\Lambda$, we see from \eqref{d:LambdaDelta} that $\Re \Lambda > -\Re \Delta$ while from \eqref{d:tildeLambda} we see that $\Re \tilde \Lambda < -\Re \Delta - 1$. Since we have that
\begin{align*}
\Re \Lambda &= -\Re (\Delta+1/2) + \Re\left(\sqrt{(\Delta + 1/2)^2 - R}\right)\\
\Re \tilde\Lambda &= -\Re (\Delta+1/2) - \Re\left(\sqrt{(\Delta + 1/2)^2 - R}\right)
\end{align*}
we see that $\Re\left(\sqrt{(\Delta + 1/2)^2 - R}\right) > 0$ and thus $|\Re \Lambda| < |\Re \tilde \Lambda|$ and thus one part of \eqref{e:Lambdabound} follows from \eqref{e:twosolutionsbound}. For the other part of \eqref{e:Lambdabound}, we estimate that
\begin{equation}
|\Lambda| = \left|\frac{R}{\sqrt{(\Delta + 1/2)^2 - R} + (\Delta + 1/2)}\right|\leq \left|\frac{R}{\Delta + 1/2}\right|
\end{equation}
where the last inequality follows since both real and imaginary parts of both summands in the denominator are positive. The fact that $\Re(\Delta+\frac{1}{2})\geq 0$ comes from the definition of our spectral domain, namely by the constraint $E^2+\eta^2-4E\leq 4\eta$.

\end{proof}

\section{Bounds on quadratic forms}

Here we obtain the necessary bounds on quadratic forms.
\begin{lemma}\label{l:Upsilonbound}
Let $\mathcal{G} = \mathcal{G}^{(\J_1)}_{(\J_2)}$ or $G^{(\J_1)}_{(\J_2)}$ for some $\J_1, \J_2$. Let
$ \Upsilon := \frac{1}{N}(\mathbb{I} - \E_{\xx} ) \mathbf{\xx}^* \mathcal{G} \mathbf{\xx} $, assuming \eqref{a:4moment}, \eqref{a:truncation} for elements of $\xx$. Then we have that
\begin{equation}\label{e:upsilonbound1} \E|\Upsilon|^{2q} \leq (Cq)^{cq}\left(\frac{\E \left (\Im \Tr \mathcal{G} \right )^{q}}{N^q(N\eta)^q} +   \frac{\E|\mathcal{G}_{11}|^{2q}}{N^q}+    \frac{\E|\mathcal{G}_{11}|^{q}}{(N\eta)^q}\right). \end{equation}   %
 Moreover, we have a more precise inequality
\begin{equation}\label{e:upsilonbound2} \E|\Upsilon|^{2q} \leq \left(\frac{Cq}{N\eta}\right)^{cq}\left(\E \left (\frac{\Im \Tr \mathcal{G} }{N}\right )^{q} +   \E\left| \frac{ \mathcal{G}_{11} }{\sqrt N}\right|^{q}\right) + \frac{(Cq)^{cq}\E (|\mathcal{G} _{12}|^{2q}+|\mathcal{G}_{11}|^{2q})}{N^q}. \end{equation}
\end{lemma}

\begin{proof}
We start by the decomposition:
\begin{equation*}
\Upsilon = \frac{1}{N} \sum \limits_{j\neq l} \overline{x_{j}} x_{l} \mathcal{G}_{jl}  + \frac{1}{N} \sum \limits_{j} (|x_{jk}|^2-1) \mathcal{G} _{jj} =\epsilon_{2} + \epsilon_{1},
\end{equation*}
where \begin{equation}\label{d:epsilonk12}\epsilon_{2} := \frac{1}{N} \sum \limits_{j\neq l} \overline{x_{j}} x_{l} \mathcal{G}_{jl}  \text{ and }
 \epsilon_{1} := \frac{1}{N} \sum \limits_{j} (|x_{j}|^2-1) \mathcal{G} _{jj}.\end{equation}

We use Rosenthal's inequality (Lemma \ref{l:rosenthal}) to obtain:
\begin{equation}\label{e:Rosenthal} \E | \epsilon_{1} |^{2q} \leq (Cq)^{2q} N^{-2q}  \left [  \sum \limits_{j} \E | x_{j} |^{4q}\E | \mathcal{G}_{jj}  |^{2q} + \left ( \mu_4 \sum \limits_{j}  \E |\mathcal{G}_{jj}  |^2 \right )^q \  \right ] .\end{equation}
We notice that
\begin{equation}\label{e:bound4thmoment}
|x_{l}| \leq D N^{1/4} \Rightarrow \E |x_{l}|^{4q} \leq D^{4q-4} N^{q-1}\mu_4,
\end{equation}
which yields that
$$ \E | \epsilon_{1} |^{2q} \leq (Cq)^{2q} N^{-q} \E|\mathcal{G}_{jj} |^{2q} .$$
For $\epsilon_{2}$ we will systematically use both Burkholder's and Rosenthal's inequalities, expressed for complex random variables in the Appendix Lemmas \ref{l:rosenthal} and \ref{l:burkholder}. Using Burkholder's Inequality we obtain
\begin{multline}
\E | \epsilon_{2} |^{2q}
\leq
 N^{-2q} (C_1q)^{2q} \left ( \E \left [ \sum \limits_{j=2}^N \left | \sum_{1\leq k\leq j-1}  {x_{k}} \mathcal{G}_{jk} \right |^2 \right ]^{q} + \max \limits_{k} \E | x_{k}|^{2q} \sum \limits_{j=2}^N \E \left | \sum_{1\leq k \leq j-1} x_{k} \mathcal{G}_{jk}  \right |^{2q} \right )\\
 + N^{-2q} (C_1q)^{2q} \left ( \E \left [ \sum \limits_{j=2}^N \left | \sum_{1\leq k\leq j-1}  {x_{k}} \mathcal{G}_{kj} \right |^2 \right ]^{q} + \max \limits_{k} \E | x_{k}|^{2q} \sum \limits_{j=2}^N \E \left | \sum_{1\leq k \leq j-1} x_{k} \mathcal{G}_{kj}  \right |^{2q} \right )
\end{multline}
We define the quantities:
\begin{equation}\label{d:Q0}
Q_0  := \sum \limits_{j=2}^N \left | \sum_{1\leq k\leq j-1}  {x_{k}} \mathcal{G}_{jk}/\sqrt N \right |^2 \quad\text{and}\quad \widehat{Q}_0  := \sum \limits_{j=2}^N \left | \sum_{1\leq k\leq j-1}  {x_{k}} \mathcal{G}_{kj}/\sqrt N  \right |^2 \end{equation}
The difficult part of the proof will be to bound expectations of powers of this quantity. For the other terms we apply Rosenthal's inequality and \eqref{e:bound4thmoment} getting that:
\begin{multline}\label{e:firststep}
\E | \epsilon_{2}|^{2q} \leq (Cq)^{2q} N^{-q} (\E |Q_0 |^q+\E |\widehat{Q}_0 |^q)
\\
+ (Cq)^{4q} N^{-\frac{3q}{2}-1}
\left [  \sum \limits_{j=2}^N  \sum_{1\leq k\leq j-1} \E | x_{l} |^{2q} (\E | \mathcal{G}_{jk}  |^{2q}+\E | \mathcal{G}_{kj}  |^{2q}) + \sum \limits_{j=2}^N \E\left\{\left ( \sum_{1\leq k\leq j-1}  | \mathcal{G}_{jk}  |^2  \right )^q + \left ( \sum_{1\leq k\leq j-1}  | \mathcal{G}_{kj}  |^2  \right )^q \right\}\right ]
 \end{multline}
For the middle term we observe that
\begin{equation}\label{e:boundoffdiag}
|\mathcal{G}_{lj} | \leq \frac{1}{2}\sqrt{\frac{\Im\mathcal{G}_{ll} }{\eta}}+ \frac{1}{2}\sqrt{\frac{\Im\mathcal{G}_{jj} }{\eta}}
\end{equation}
which can be obtained as follows. Let $u_j$ be the $j$th normalized eigenvector of $\mathcal{G} $ and $\lambda_0= 0$. Then
\begin{multline}
|\mathcal{G}_{lj} |= \left|\sum_{q=0}^N \frac{u_{lq}u_{qj}}{\lambda_q - z}\right| \leq \sum_{q=0}^N \frac{|u_{lq}u_{qj}|}{|\lambda_q - z|}
\leq \frac{1}{2}\sum_{q=0}^N \frac{|u_{lq}|^2 + |u_{qj}|^2}{|\lambda_q - z|} \leq \frac{1}{2}\sqrt{\sum_{q=0}^N \frac{|u_{lq}|^2} {|\lambda_q - z|^2}} + \frac{1}{2}\sqrt{\sum_{q=0}^N \frac{|u_{qj}|^2} {|\lambda_q - z|^2}}\\
\end{multline}
where in the last step we recall that the eigenvectors are normalized and use Jensen's ineqality. Then \eqref{e:boundoffdiag} follows.
Using for the last term that $\sum_{l=1}^{N} | \mathcal{G}_{jl}  |^2 \leq \eta^{-1} \mathrm{Im} \ \mathcal{G}_{jj}$ we now get that
\begin{equation}\label{e:epsilon2bound}
\E | \epsilon_{2} |^{2q} \leq (Cq)^{2q} N^{-q}   \E |Q_0 |^q + \frac{(Cq)^{4q}}{ (N\eta)^{q}}  \E | \Im\mathcal{G}_{ll} |^{q}
.
\end{equation}

We will now bound the quantity $\E|Q_0 |^q$, and note that $\E|\widehat{Q}_0 |^q$ is similar. We will implement an induction scheme on the quantity $\E|Q_0 |^q$ to gradually decrease its exponent $q$ and finally remove it.
The technique is similar to one in \cite{gotze2016optimal, gotze2014rate} but we expand Rosenthal and Burkholder inequalities to complex entries in the Appendix and improve the bound so that it can be applied to the soft edge.
Similar to \cite{gotze2016optimal} we define the quantities
\begin{multline}\label{d:Qs} Q_\nu  := \sum \limits_{j=2}^N \left | \sum_{1\leq k\leq j-1} x_{j} a_{jk}^{(\nu)} \right |^2,
\quad
 Q_{\nu1}  := \sum \limits_{l=1}^N a_{ll}^{(\nu+1)} ,
 \\
 Q_{\nu2}  := \sum \limits_{l=1}^N \left ( |x_{l}|^2 -1 \right ) a_{ll}^{(\nu+1)}, \quad \text{and} \quad
 Q_{\nu3}  := \sum \limits_{1 \leq l \neq j \leq N} x_{l}\overline{x_{j}} a_{lj}^{(\nu+1)}\\
 ,\end{multline}
where $a_{lj}^{(\nu)}$ are defined recursively via
\begin{equation}\label{d:as}
a_{lj}^{(0)}:= \frac{1}{\sqrt{N}} \mathcal{G}_{lj}  \quad
\text{and}
\quad a_{rl}^{(\nu+1)} := \sum \limits_{j= \max\{r, l\}+1}^{N}  a_{rj}^{(\nu)}\overline{a_{lj}^{(\nu)}}\end{equation}
for $\nu = 0,1,...,L-1, L$ with integer $L$ such that $q = 2^{L}$. From \cite{gotze2016optimal} Lemma 5.1 and Corollaries 5.2 and 5.3, we have the following bounds:
\begin{equation}
\label{e:arr}\max\{|a_{rr}^{(\nu+1)}|,\; \sum_j |a_{jr}^{( \nu)}|^2\}\leq \left(\frac{\Im \Tr \mathcal{G}  }{N\eta}\right)^{2^\nu-1}\, \frac{\Im \mathcal{G}_{rr} }{N\eta}\end{equation}
\medskip
To set up our induction scheme we expand the absolute value square and interchange the order of summations in $Q_\nu $:
\begin{multline}\label{e:decomp}
Q_{\nu}  = \sum \limits_{j=2}^N  \sum_{1 \leq k_1, k_2 \leq j-1}x_{k_1} \overline{x_{k_2}}  a_{k_1j}^{(\nu)}\overline{a_{k_2j}^{(\nu)}}
=
 \sum_{1 \leq k_1, k_2 \leq N-1}x_{k_1} \overline{x_{k_2}} \sum \limits_{j= \max\{k_1, k_2\}+1}^{N}  a_{k_1j}^{(\nu)}\overline{a_{k_2j}^{(\nu)}}
  \\=
   \sum_{1 \leq j_1, j_2 \leq N}x_{j_1} \overline{x_{j_2}}   a_{j_1j_2}^{(\nu+1)}
   = Q_{\nu 1}  + Q_{\nu 2}  + Q_{\nu 3}
\end{multline}
Now taking power $2^{L-\nu}$ and expectation we obtain that
\begin{equation}
\E | Q_{\nu}  |^{2^{L-\nu}} \leq 3^{2^{L-\nu}} \left ( \E | Q_{\nu1}  |^{2^{L-\nu}} + \E | Q_{\nu2}  |^{2^{L-\nu}} + \E | Q_{\nu3}  |^{2^{L-\nu}} \right )\\
\end{equation}

We apply Rosenthal's inequality for the quantity
$ Q_{\nu 2} $ getting that
\begin{multline}
\E | Q_{\nu2}  |^{2^{L-\nu}}
\leq  (Cq)^q \left [\E \left ( \sum \limits_{l=1}^N  | a_{ll}^{(\nu+1)} |^2  \right )^{2^{L-(\nu+1)}} + \sum \limits_{l=1}^N \E | x_{l}^2 -1 |^{2^{L-\nu}} |a_{ll}^{(\nu+1)} |^{2^{L-\nu}} \right ]
\\
\leq (Cq)^q \left [ \E \left ( \sum \limits_{l=1}^N | a_{ll}^{(\nu+1)} |^2 \right )^{2^{L-(\nu+1)}} + N^{2^{L-(\nu+1)}} \frac{1}{N} \sum \limits_{l=1}^N  \E |a_{ll}^{(\nu+1)} |^{2^{L-\nu}} \right ] .
\end{multline}
where we used \eqref{e:bound4thmoment} in the last line. We notice that the first term is bounded above by the second term by Jensen's Inequality, so we obtain that
\begin{equation}
\E | Q_{\nu2}  |^{2^{L-\nu}}\leq (Cq)^q N^{2^{L-(\nu+1)}} \frac{1}{N} \sum \limits_{l=1}^N  \E |a_{ll}^{(\nu+1)} |^{2^{L-\nu}}
\leq Cq^q J_\nu,
\end{equation}
where we use \eqref{e:arr} and introduce the notation
\begin{equation}
J_{\nu}: = N^{-2^{L- (\nu+1)}}\E\left| \frac{\Im \Tr \mathcal{G}  }{N\eta} \right|^{2^L- 2^{L-\nu}}\left|\frac{\Im \mathcal{G}_{11} (z)}{\eta}\right|^{2^{L-\nu}}.
\end{equation}

Now we apply Burkholder's Inequality to $\E|Q_{\nu3}|^q$ we obtain a bound which involves $\E|Q_{\nu+1}|^{q/2}$ and $\E|\widehat{Q}_{\nu+1}|^{q/2}$ as in \eqref{e:firststep}. We use Rosenthal's inequality to bound the other term arising from the application of Burkholder's Inequality:
\begin{multline}
\E|Q_{\nu3}|^{2^{L-\nu}} \leq \E|\sum_{1 \leq j_1, j_2 \leq N, j_1 \neq j_2}x_{j_1} \overline{x_{j_2}}   a_{j_1j_2}^{(\nu+1)}|^{2^{L-\nu}}
\\
\leq
 (Cq)^{q} (\E |Q_{\nu+1} |^{2^{L-(\nu+1)}} + \E |\widehat{Q}_{\nu+1} |^{2^{L-(\nu+1)}})+
 \E|x_1|^{2^{L-\nu}} \sum \limits_{j=2}^n
 \E\left( \left | \sum \limits_{k=1}^{j-1} a_{jk}^{(\nu+1)} x_k \right |^{2^{L-\nu}} + \left | \sum \limits_{k=1}^{j-1} a_{kj}^{(\nu+1)} x_k \right |^{2^{L-\nu}}\right)
 \\
\leq
 (Cq)^{q} (\E |Q_{\nu+1} |^{2^{L-(\nu+1)}} + \E |\widehat{Q}_{\nu+1} |^{2^{L-(\nu+1)}}) + (Cq)^{2q} N^{2^{L-(\nu+2)}} \frac{1}{N} \sum \limits_{j=2}^N \E\left ( \sum_{1\leq l\leq N, l \neq j} | a_{jl}^{(\nu+1)}|^2 \right )^{2^{L-(\nu+1)}}
\\
+ (Cq)^{q}  N^{2^{L-(\nu+1)}} \frac{1}{N^2} \sum \limits_{j=2}^N \sum_{1\leq l\leq N, l \neq j} \E | a_{jl}^{(\nu+1)} |^{2^{L-\nu}}
\end{multline}
The resulting terms are bounded by \eqref{e:arr} and the following argument. By H\"older inequality and above definition, we obtain that
\begin{multline}\label{e:arjhighpower}
\frac{1}{N^2} \sum \limits_{j=2}^N \sum_{1\leq r\leq N, r\neq j}|a_{rj}^{(\nu+1)}|^{2^{L-\nu}}
\leq
\frac{1}{N^2} \sum \limits_{j=2}^N  \sum_{1\leq r\leq N, r\neq j} \left|\sum_{l_1} |a_{r{l_1}}^{(\nu)}|^2 \sum_{l_2} |a_{{l_2}j}^{(\nu)}|^2\right|^{2^{L-(\nu+1)}} \\
 =
  \frac{1}{N^2} \sum \limits_{j=2}^N  \sum_{1\leq r\leq N, r\neq j} \left|a_{rr}^{(\nu)}a_{jj}^{( \nu)}\right|^{2^{L-(\nu+1)}} \leq \left(\frac{1}{N}\sum_j |a_{jj}^{(\nu)}|^{2^{L-(\nu+1)}}\right)^2
   \leq
    \frac{1}{N}\sum_j |a_{jj}^{(\nu)}|^{2^{L-\nu}}
\end{multline}
This yields that
\begin{multline}
\E|Q_{\nu3}|^{2^{L-\nu}}
\leq
(Cq)^{q} (\E |Q_{\nu+1} |^{2^{L-(\nu+1)}} + \E |\widehat{Q}_{\nu+1} |^{2^{L-(\nu+1)}})
\\
+
 (Cq)^{2q} N^{2^{L-(\nu+2)}} \frac{1}{N} \E \left |  \left(\frac{\Im \Tr \mathcal{G}  }{N\eta}\right)^{2^{\nu+1}-1} \frac{\Im \mathcal{G}_{11} }{N\eta} \right |^{2^{L-(\nu+1)}}
+ (Cq)^{q}  N^{2^{L-(\nu+1)}} \E   \left|\left(\frac{\Im \Tr \mathcal{G}  }{N\eta}\right)^{2^{\nu}-1}\, \frac{\Im \mathcal{G}_{11} }{N\eta}
\right|^{2^{L-\nu}}
\\
\leq
(Cq)^{q} (\E |Q_{\nu+1} |^{2^{L-(\nu+1)}} + \E |\widehat{Q}_{\nu+1} |^{2^{L-(\nu+1)}})+
 (Cq)^{2q}(J_{\nu+1} + J_{\nu}).
\end{multline}

\medskip
Same bounds hold for $\widehat{Q}_{\nu}$ with $\widehat{Q}_{\nu1}$ and $\widehat{Q}_{\nu2}$ defined analogously, thus by an induction argument we obtain that
\begin{multline}
\E(|Q_0|^{2^L}+|\widehat{Q}_0|^{2^L})
 \\
 \leq (Cq)^{cq}\left(\E(|Q_{L}|+|\widehat{Q}_{L}|) + \sum_{\nu=0}^{L-1} \left(\E(|Q_{\nu1}|^{2^{L-\nu}}+|\widehat{Q}_{\nu1}|^{2^{L-\nu}})+ \E(|Q_{\nu2}|^{2^{L-\nu}}+|\widehat{Q}_{\nu2}|^{2^{L-\nu}})
 + J_\nu + J_{\nu+1}\right)\right)
 \\
 \leq
  (Cq)^{cq}\left(\E|Q_{L}| + \sum_{\nu=0}^{L-1} \sum_{\nu=0}^{L-1} (\E(|Q_{\nu1}|^{2^{L-\nu}}+|\widehat{Q}_{\nu1}|^{2^{L-\nu}})+ \sum_{\nu=0}^{L}
 J_\nu\right).
\end{multline}

Now, we bound the three terms on the RHS. For $0 \leq \nu \leq L$ we have that
\begin{multline}
J_\nu\leq \eta^{-2^L}\E\left( \frac{\Im \Tr \mathcal{G}  }{N} \right)^{2^L}\left|\frac{\Im \mathcal{G} _{11}(z)}{\sqrt N}\, \frac{N}{\Im \Tr \mathcal{G}  }\right|^{2^{L-\nu}}
\\
\leq \eta^{-2^L}\E\left(\left( \frac{\Im \Tr \mathcal{G}  }{N}\right)^{2^{L}} + \left|\frac{\Im \mathcal{G} _{11}(z)}{\sqrt N}\right|^{2^{L}}\right)
\end{multline}

By the definition of $Q_{\nu1} $ and the definition of the coefficients $a_{ll}^{(\nu+1)}$, we can check that
\beq \label{Q1}
\E | Q_{\nu1}  |^{2^{L-\nu}} \leq \E \left [\frac{\Im \Tr \mathcal{G}  }{N\eta}  \right ]^{2^L} ,
\eeq
and lastly, using \eqref{e:decomp} and \eqref{e:arr} we obtain
\begin{equation}
\E |Q_L| = \E  \sum_{1 \leq j_1, j_2 \leq N}x_{j_1} \overline{x_{j_2}}   a_{j_1j_2}^{(\nu+1)} = \E\sum_j a_{jj}^{(L+1)} \leq \E\left(\frac{\Im \Tr \mathcal{G}  }{N\eta}\right)^{2^L},
\end{equation}
\end{proof}
and similar for $\E |\widehat{Q}_L|$, which upon substitution yields that
\begin{equation}\label{e:Q0bound}
\E |Q_0|^q +\E |\widehat{Q}_0|^q\leq (Cq)^{cq}\eta^{-q}\left( \E\left(\frac{\Im \Tr \mathcal{G}  }{N}\right)^{q} + \E\left|\frac{\Im \mathcal{G} _{11}(z)}{\sqrt N}\right|^{q}\right)
\end{equation} and the desired bounds on $\E |\Upsilon|^{2q}$ follows.

\section{Non-optimal bound and bootstrap argument in the bulk}

Let
\begin{equation}\label{d:lambda}
\lambda^{(\J_1)}_{(\J_2)} : = \max \{|\Lambda^{(\J_1)}_{(\J_2)}|\chi_{S_{E,\eta}}, \min\{|\Lambda^{(\J_1)}_{(\J_2)}|, |\tilde \Lambda^{(\J_1)}_{(\J_2)}|\}, |\Im \Lambda^{(\J_1)}_{(\J_2)}|\}\end{equation}
By Proposition \ref{b-R},
$
|\theta|^q\E |\lambda^{(\J_1)}_{(\J_2)}|^{2q} \leq C^{2q}|\theta|^q \E| R^{(\J_1)}_{(\J_2)}|^q.
$
Taking expectation of a power of $\theta R$ we obtain (as in \cite{cacciapuoti2015bounds})
\begin{multline}\label{e:calc}\E|\theta R^{(\J_1)}_{(\J_2)}|^q \leq
 \frac{|\theta |^q}{N} \sum_{k=1}^N \E\left|\left(T_k
  + \Upsilon^{(\J_1\cup \{k\})}_{(\J_2)}\right)G^{(\J_1)}_{(\J_2), kk}\right|^q
  \\
\leq \left|\theta \right|^q \E\left|\left(T_1
  + \Upsilon^{(\J_1 \cup \{1\})}_{(\J_2)}\right)G^{(\J_1)}_{(\J_2), 11}\right|^q
  \\
\leq \frac{\E \left |C \theta (|\,|\J_1|-|\J_2|\,|+1)G^{(\J_1)}_{(\J_2),11}\right|^q}{(N\eta)^q} + \left|C\theta \right|^q\sqrt{\E|G^{(\J_1)}_{(\J_2), 11}|^{2q}\E\left|
  \Upsilon^{(\J_1 \cup \{1\})}_{(\J_2)}\right|^{2q}}
  \\
\leq \frac{\E \left| C\theta (|\,|\J_1|-|\J_2|\,|+1)G^{(\J_1)}_{(\J_2), 11}\right|^q}{(N\eta)^q} + \left|C\theta q\right|^{cq}\sqrt{ \E|G^{(\J_1)}_{(\J_2),11}|^{2q}\frac{\E|\mathcal{G}^{(\J_1 \cup \{1\})}_{(\J_2),22}|^{2q}}{N^q}}
\\
+
 \left|C\theta q\right|^{cq}\sqrt{\E|G^{(\J_1)}_{(\J_2),11}|^{2q}   \left(\E   \frac{( \Im \Tr \mathcal{G}^{(\J_1 \cup \{1\})}_{(\J_2)})^q}{(N\eta)^qN^q} +  \frac{\E|\mathcal{G}^{(\J_1 \cup \{1\})}_{(\J_2),22}|^{q}}{(N\eta)^q}      \right)}
\\
\leq \frac{\E \left| C\theta (|\,|\J_1|-|\J_2|\,|+1) G^{(\J_1)}_{(\J_2), 11}\right|^q}{(N\eta)^q} + \left|C\theta q\right|^{cq}\sqrt{ \E|G^{(\J_1)}_{(\J_2),11}|^{2q}\frac{\E|\mathcal{G}^{(\J_1 \cup \{1\})}_{(\J_2),22}|^{2q}}{N^q}}
\\
+
 \left|C\theta \right|^q\sqrt{\E|G^{(\J_1)}_{(\J_2),11}|^{2q}
\frac{(Cq)^{cq}}{(N\eta)^q}   \left(\left(\frac{\left|\,|\J_1|+1-|\J_2|\,\right|}{N\eta}\right)^{q} +  \E \left(\Im \Delta +  \Im \Lambda^{(\J_1 \cup \{1\})}_{(\J_2)}\right)^q +   \E|\mathcal{G}^{(\J_1 \cup \{1\})}_{(\J_2),22}|^{q}\right)}
\\
\leq \left|C\theta q\right|^{cq}\sqrt{ \E|G^{(\J_1)}_{(\J_2),11}|^{2q}\frac{\E|\mathcal{G}^{(\J_1 \cup \{1\})}_{(\J_2),22}|^{2q}}{N^q}}
\\
+
 \left|Cq\right|^{cq}\frac{|\theta| ^{\frac{q}{4}}\sqrt{ \E| \sqrt \theta G^{(\J_1)}_{(\J_2),11}|^{2q}}} {(N\eta)^{\frac{q}{2}}} \left(\left|\,|\J_1|+1-|\J_2|\,\right|^q +  \sqrt{\E|\sqrt \theta\lambda|^q} + \sqrt{\E|\sqrt \theta\mathcal{G}^{(\J_1 \cup \{1\})}_{(\J_2),22}|^{q}}\right)
\end{multline}
In the second to last line, the term $\frac{|\J_1|+1-|\J_2|}{N\eta}$ arises from equation \eqref{e:traces}, and $\Im \Lambda^{(\J_1 \cup \{1\})}_{(\J_2)}$ is close to $\Im \Lambda^{(\J_1)}_{(\J_2)}$ similar to \eqref{e:minorcalc}. Since for any $x, \delta>0$, $x^{1/4} < \delta x + \delta^{-1/3}$, setting $\delta = \left(2(\E|\sqrt \theta G^{(\J_1)}_{(\J_2),11}|^{2q})^{1/2}\frac{(Cq(|\,|\J_1|-|\J_2|\,|+1))^{cq}|\theta|^{q/4}}{(N\eta)^{q/2}}\right)^{-1}$ and using Cauchy-Schwarz inequality on $ \sqrt{\E|\sqrt \theta\lambda|^q}$, we get
\begin{multline}\label{e:bb}
	|\theta |^q\E \lambda^{2q}
\leq \frac{(Cq(|\,|\J_1|-|\J_2|\,|+1))^{cq}}{(N\eta)^{q/6}} |\theta|^{q/12} \left((\E|\sqrt \theta \mathcal{G}^{(\J_1 \cup \{1\})}_{(\J_2),22}|^{2q})^{1/2}+\left|\,|\J_1| -|\J_2|+1\right|^{\frac{q}{2}}\right)\left((\E|\sqrt \theta  G^{(\J_1)}_{(\J_2),11}|^{2q})^{1/2} + 1 \right) \\
+\left|C\theta q\right|^{cq}\sqrt{ \E|G^{(\J_1)}_{(\J_2),11}|^{2q}\frac{\E|\mathcal{G}^{(\J_1 \cup \{1\})}_{(\J_2),22}|^{2q}}{N^q}}.\end{multline}

\begin{lemma} \label{l:Gkk}
Let $q < \left (\frac{N \eta}{|\sqrt{\theta}|} \right )^{1/4}$, $N\eta>|\sqrt{\theta}| M$ for some constant $M>0$, fixed $E$. Assume that $\J_1, \J_2$ are such that $0 \leq | \J_1| - |\J_2| \leq  Cq$ for a uniform constant $C$. Then with definitions as before, $\E|G^{(\J_1)}_{(\J_2), 11}\sqrt \theta|^{q} \leq C^{q}$ and $\E|\mathcal{G}^{(\J_1 \cup \{1\})}_{(\J_2), 22}\sqrt \theta|^{q} \leq C^{q}$, for some constant $C$.
\end{lemma}

\begin{proof}
We will implement an induction argument similar to \cite{cacciapuoti2015bounds, gotze2016optimal}.  The induction hypothesis will be that for $\eta_j = \eta_0/16^j$ for some constant $\eta_0$, any $\J_{1, j}, \J_{2, j}$ with $|\J_{1, j}|= |\J_{2, j}| \leq |\log_{16} \eta|+1-j = :L_j$ and $k \notin \J_{1, j}$
\begin{equation}\label{e:inductionhyp}
\E|G^{(\J_1 \cup \J_{1, j})}_{(\J_2 \cup \J_{2, j}), 11}(\eta_j)\sqrt \theta|^{q} < C_0^{q} \quad \text{and}\quad \E|\mathcal{G}^{(\J_1 \cup\J_{1, j} \cup \{1\})}_{(\J_2\cup\J_{2, j}), 22}(\eta_j)\sqrt \theta|^{q} < C_0^{q}
 \end{equation}
 for $q < \left (\frac{N \eta_j}{|\sqrt{E}|} \right )^{1/4}$ for a universal constant $C_0$. We notice that this holds to initiate our induction for $\eta_0$ constant. Letting $\eta_{j+1} = \eta_j/16$ and $L_{j+1} = L_j - 1$ we will show that inequality \eqref{e:inductionhyp} taken at $\eta_j$ implies the same inequality with the same constant $C_0$ for $\eta_{j+1}$. For easier notation, we will suppress the dependence on $\J_1, \J_2$ mentioning only the step where they come up (which is equation \eqref{e:lambdabootstrap}).

 From the induction hypothesis and Lemma \ref{G11eta/s} we see that
\begin{equation}\begin{split}\label{e:s}
\E|G^{(\J_{1, j})}_{(\J_{2, j}), 11}(\eta_{j+1})\sqrt \theta|^{q} < (16C_0)^{q} \quad \text{and}\quad \E|\mathcal{G}^{(\J_{1, j} \cup \{1\})}_{(\J_{2, j}), 22}(\eta_{j+1})\sqrt \theta|^{q} < (16C_0)^{q}
 \end{split}\end{equation}
for any $\J_{1, j}, \J_{2, j}$ with $|\J_{1, j}|= |\J_{2, j}| \leq L-j$. This will need to be improved to the bound $C_0^q$ for any $\J_{1, j+1}, \J_{2, j+1}$ of size up to $L - j -1$.

From \eqref{e:Gkk} and \eqref{d:tildeLambda} we obtain that
\begin{align}
&G^{(\J_{1,j+1})}_{(\J_{2,j+1}), 11} = \Delta - \sqrt{\theta}\Delta(\sqrt{\theta}\Lambda^{(\J_{1,j+1})}_{(\J_{2,j+1})}-\sqrt{\theta} T_1 - \sqrt{\theta} \Upsilon^{(\J_{1,j+1}\cup \{1\})}_{(\J_{2,j+1})}) G^{(\J_{1,j+1})}_{(\J_{2,j+1}), 11}  \\
&G^{(\J_{1,j+1})}_{(\J_{2,j+1}), 11} = \Delta - \sqrt{\theta}\Delta(-\sqrt{\theta}\tilde\Lambda^{(\J_{1,j+1})}_{(\J_{2,j+1})})-\sqrt{\theta}T_1 - \sqrt{\theta} \Upsilon^{(\J_{1,j+1}\cup \{1\})}_{(\J_{2,j+1})}) G^{(\J_{1,j+1})}_{(\J_{2,j+1}), 11} + \theta\Delta(2\Delta + 1)G^{(\J_{1,j+1})}_{(\J_{2,j+1}), 11}.\end{align}
The analogous statements for $\mathcal{G}^{(\J_{1, j} \cup \{1\})}_{(\J_{2, j}), 22}$ follow similarly from \eqref{e:curlyGkk} and \eqref{d:tildeLambda}:
\begin{equation}\label{e:bootstrapcalG}
\mathcal{G}^{(\J_{1, j+1} \cup \{1\})}_{(\J_{2, j+1}), 22} = \Delta - \sqrt{\theta}\Delta\left[\sqrt{\theta}\Lambda^{(\J_{1, j+1}\cup \{1\})}_{(\J_{2, j+1}), 11}-\sqrt{\theta}\mathcal{T}_1 - \sqrt{\theta} Y^{(\J_{1, j+1}\cup \{1\} )}_{(\J_{2, j+1}\cup \{2\})}\right] \mathcal{G}^{(\J_{1, j+1}\cup \{1\})}_{(\J_{2, j+1}), 22} .\end{equation}

This yields that
\begin{align*}
&|G^{(\J_{1,j+1})}_{(\J_{2,j+1}), 11}|\leq |\Delta| + |G^{(\J_{1,j+1})}_{(\J_{2,j+1}), 11}|(|\sqrt{\theta}
\Lambda^{(\J_{1,j+1})}_{(\J_{2,j+1})}|+|\sqrt{\theta}T^{(\J_{1,j+1}\cup \{1\})}_{(\J_{2,j+1})}|+|\sqrt{\theta}\Upsilon^{(\J_{1,j+1}\cup \{1\})}_{(\J_{2,j+1})}|)|\sqrt{\theta} \Delta|
\\
&|G^{(\J_{1,j+1})}_{(\J_{2,j+1}), 11}|\leq \left|\frac{\Delta}{1-\theta \Delta(2\Delta +1)}\right| + |G^{(\J_{1,j+1})}_{(\J_{2,j+1}), 11}|(|\sqrt{\theta}\tilde\Lambda^{(\J_{1,j+1})}_{(\J_{2,j+1})}|+|\sqrt{\theta}T^{(\J_{1,j+1}\cup \{1\})}_{(\J_{2,j+1})}|+|\sqrt{\theta}\Upsilon^{(\J_{1,j+1}\cup \{1\})}_{(\J_{2,j+1})}|)|\sqrt{\theta} \Delta|
\end{align*}
and using \eqref{e:De} we see that $1-\theta \Delta(2\Delta +1) = -\theta\Delta^2$ and thus $\frac{\Delta}{1-\theta \Delta(2\Delta +1)}= \frac{1}{\theta \Delta}$.

We will use the bounds $C_1 \leq|\sqrt{\theta}\Delta| \leq C_2$, valid in our domain, and let $C = \max \{C_1, C_2\}$. So, we have that:
\begin{equation*}
|\sqrt{\theta}G^{(\J_{1,j+1})}_{(\J_{2,j+1}), 11}|
\leq C \left [1+|\sqrt{\theta}G^{(\J_{1,j+1})}_{(\J_{2,j+1}), 11}|(|\sqrt{\theta}|\min\{|\Lambda^{(\J_{1,j+1})}_{(\J_{2,j+1})}|, |\tilde\Lambda^{(\J_{1,j+1})}_{(\J_{2,j+1})}|\}+|\sqrt{\theta}T_1|+|\sqrt{\theta}\Upsilon^{(\J_{1,j+1}\cup \{1\})}_{(\J_{2,j+1})}|) \right ]
\end{equation*}

and, taking power $q$, expectation, and using Cauchy-Schwarz we get at $\eta_{j+1}$

\begin{multline} \label{3.5}
\mathbb{E}|\sqrt{\theta}G^{(\J_{1,j+1})}_{(\J_{2,j+1}), 11}|^q
\leq C^q \big[ 1+ \sqrt{\mathbb{E}|\sqrt{\theta}G^{(\J_{1,j+1})}_{(\J_{2,j+1}), 11}|^{2q}} \sqrt{\mathbb{E}(|\sqrt{\theta}\lambda^{(\J_{1,j+1})}_{(\J_{2,j+1})}|\})^{2q}}
\\
+ \frac{|\sqrt{\theta}|^q }{(N\eta_{j+1})^q} \mathbb{E}|\sqrt{\theta}G^{(\J_{1,j+1})}_{(\J_{2,j+1}), 11}|^{q} +\sqrt{\mathbb{E}|\sqrt{\theta}G^{(\J_{1,j+1})}_{(\J_{2,j+1}), 11}|^{2q}} \sqrt{\mathbb{E}|\sqrt{\theta}\Upsilon^{(\J_{1,j+1}\cup \{1\})}_{(\J_{2,j+1})}|^{2q}} \big]
\end{multline}
Using the above, Lemma \ref{l:Upsilonbound}, and a calculation similar to \eqref{e:calc} we obtain again at $\eta_{j+1}$
\begin{multline}\label{e:bootstrap}
 \mathbb{E}|\sqrt{\theta}G^{(\J_{1,j+1})}_{(\J_{2,j+1}), 11}|^q \leq (Cq)^{cq} \Bigg [ 1+ \sqrt{\mathbb{E}|\sqrt{\theta}G^{(\J_{1,j+1})}_{(\J_{2,j+1}), 11}|^{2q}} \sqrt{\mathbb{E}|\sqrt{\theta} \lambda^{(\J_{1,j+1})}_{(\J_{2,j+1})}|^{2q} }
 +
  \frac{|\sqrt{\theta}|^q} {(N\eta_{j+1})^q} \mathbb{E}|\sqrt{\theta}G^{(\J_{1,j+1})}_{(\J_{2,j+1}), 11}|^{q}
  \\
  + \sqrt{\mathbb{E}|\sqrt{\theta}G^{(\J_{1,j+1})}_{(\J_{2,j+1}), 11}|^{2q}}   \frac{|\sqrt{\theta}|^{q/4}}{(N\eta_{j+1})^{q/2}} \sqrt{1 + \E |\sqrt \theta\lambda^{(\J_{1,j+1})}_{(\J_{2,j+1})}| ^{q} +  \E|\sqrt \theta \mathcal{G}^{(\J_{1,j+1})\cup\{1\}}_{(\J_{2,j+1}), 22}|^{q}}
  \\
  +  \sqrt{\E|\sqrt{\theta}G^{(\J_{1,j+1})}_{(\J_{2,j+1}), 11}|^{2q}\frac{\E|\sqrt \theta\mathcal{G}^{(\J_{1,j+1})\cup\{1\}}_{(\J_{2,j+1}), 22}|^{2q}}{N^q}}\Bigg ]
\end{multline}

We use \eqref{e:s} to bound the terms $\mathbb{E}|\sqrt{\theta}G^{(\J_{1,j+1})}_{(\J_{2,j+1}), 11}|^{2q}$ and $\E| \sqrt \theta\mathcal{G}^{(\J_{1,j+1})\cup\{1\}}_{(\J_{2,j+1}), 22}|^{2q}$ in the above inequality, noting that $|\J_{1,j+1}\cup\{1\}| \leq L_j$. To use \eqref{e:s} we need $2q \leq  \left (\frac{N \eta_j}{|\sqrt{E}|} \right )^{1/4}$, which gives us $q \leq \left(\frac{N \eta_j}{16|\sqrt{E}|} \right )^{1/4} $, which is what we need. Then using \eqref{e:s} on equation \eqref{e:bb} at $\eta_{j+1}$ and recalling that $|\,|\J_1| - |\J_2|\,| \leq Cq$ we obtain

\begin{equation}\label{e:lambdabootstrap}
\E | \sqrt{\theta} \lambda^{(\J_{1,j+1})}_{(\J_{2,j+1})} |^q
\leq  \frac{(Cq)^{cq}}{(N\eta_{j+1})^{q/6}} |\theta|^{q/12} \left((16C_0)^{q} + (Cq)^{q/2} \right)^2 + \frac{(Cq)^{cq}(16C_0)^{2q}}{N^{q/2}}
\end{equation}

Substituting this into \eqref{e:bootstrap}, we obtain that at $\eta_{j+1}$:
\begin{multline}
\E |\sqrt{\theta}G^{(\J_{1,j+1})}_{(\J_{2,j+1}), 11}|^q \leq (Cq)^{cq} \Bigg [ 1 + (16C_0)^q\frac{|\theta|^{q/12}} {(N\eta_{j+1})^{q/6}} \left((16C_0)^{q} + (Cq)^{cq} \right)+ \frac{(16C_0)^{3q}}{N^{q/2}}
\\
 +
 (16C_0)^q   \frac{q^{cq}|\sqrt{\theta}|^{q/4}}{(N\eta_{j+1})^{q/2}} \sqrt{2 + \frac{(Cq)^{cq}}{(N\eta_{j+1})^{q/6}} |\theta|^{q/12} \left((16C_0)^{q} + (Cq)^{q/2} \right)^2+ (16C_0)^{q}  }
\Bigg ] \\
\leq (Cq)^{cq} \left [ 2 + K^q \left( \frac{ |\sqrt{\theta}| } {N\eta} \right)^{q/6}  \right ]\end{multline}
for a constant $K>0$ depending on $C_0$ and $C$. We can choose $C_0>2C$ and $\frac{ N \eta}{|\sqrt{\theta}|} > M > K^{6}$, so that $ K^q \left( \frac{|\sqrt{\theta}|} {N\eta} \right)^{q/6} <1$ and therefore $\E[\sqrt{\theta}G_{11}(\eta_{j+1})]^q<C_0^q$ as required. We notice that all the steps are identical for $\mathcal{G}^{(\J_1 \cup \{1\})}_{(\J_2), 22}$ using  \eqref{e:bootstrapcalG}, and exactly one row gets stripped as well as exactly one column so $|\J_{1, j+1}| = |\J_{2, j+1}|$.

\end{proof}

\section{Optimal Bound for the Stieltjes transform}

In this section we prove Theorem \ref{t:sti}. We will use the matrix expansion algorithm from \cite{cacciapuoti2015bounds}, which carries over directly as it is based entirely on linear algebra of resolvents. We will make a note of the important modifications. We note, importantly, that as we expand resolvent entries, we will be removing columns of $X_N$ and we never need to remove rows. The expansion algorithm yields results in terms of high moments of the following quantities:

\begin{equation}\label{e:quantities} | \sqrt{\theta} G_{kk}^{(\J)} | , \left | \frac{1}{\sqrt{\theta}G_{kk}^{(\J)}} \right | ,\left| (\mathbb{I}-\E_k) \frac{1}{\sqrt{\theta} G_{kk}^{(\J)}} \right |, |\sqrt{\theta} G_{kl}^{(\J)}|,  \end{equation}
and we begin this section by estimating these moments.
\\

To obtain optimal bounds on $\Lambda$ near the soft edge the fluctuation bound on relevant quadratic forms \eqref{e:upsilonbound1} needs to be improved. For that purpose we will use \eqref{e:upsilonbound1} to obtain bounds on $|\sqrt{\theta} G_{(\J_2), kl}^{(\J_1)}|$ as well as $|\sqrt{\theta} \mathcal{G}_{(\J_2), kl}^{(\J_1)}|$ then use these in \eqref{e:upsilonbound2} to improve on the RHS of \eqref{e:upsilonbound2}. For convenience of notation we introduce the control parameter
\begin{equation}\label{d:controlparam} \mathcal{E}_q := \frac{1}{N^{q} |\theta|^{q/2}} + \max \left \{\frac{[\Im (|\theta| \Delta )]^{q} + \E |\theta \Lambda|^q}{(N\eta)^q} , \frac{|\theta|^{q}}{(N\eta)^{2q}} \right \}  .\end{equation}

We now show how to estimate the last quantity in \eqref{e:quantities}, using the formulas (see e.g. (2.20) of \cite{pillai2014universality}) (valid also for any $\J_1, \J_2$, with $k, l \notin \J_1\cup\J_2$)
\begin{equation} \label{e:form-off-diag}\begin{split}
&\sqrt{\theta} G_{kl} = 	\sqrt{\theta} G_{ll} \sqrt{\theta} G_{kk}^{(\{l\})} (\sqrt{\theta} (\mathbf{x}^{k}/\sqrt N)^* \mathcal{G}^{(\{k, l\}) }(\mathbf{x}^{l}/\sqrt N)) =: \sqrt{\theta} G_{ll} \sqrt{\theta} G_{kk}^{(\{l\})} K_{kl}
\\
&\sqrt{\theta} \mathcal{G}_{kl} = 	\sqrt{\theta} \mathcal{G}_{ll} \sqrt{\theta} \mathcal{G}_{(\{l\}),kk} (\sqrt{\theta} (\mathbf{x}_{k}/\sqrt N) G_{(\{k, l\}) }(\mathbf{x}_{l}/\sqrt N))^* =: \sqrt{\theta} \mathcal{G}_{ll} \sqrt{\theta} \mathcal{G}_{(\{l\}), kk} \mathcal{K}_{kl} .
\end{split}\end{equation}

We can define $K^{(\J_1)}_{(\J_2), kl}, \mathcal{K}^{(\J_1)}_{(\J_2), kl}$ analogously. The following lemma provides the necessary bound on $\E |  K_{kl} |^{2q}$ and an improved bound on $\Upsilon_{(\J_2)}^{(\J_1)}$.

 \begin{lemma} \label{off-diag}

Assume \eqref{a:4moment} and \eqref{a:truncation} for the entries of the matrix $X_N$ as before and let $\theta=E+i\eta$. Then there exist constants $c,c_0,C,M_1,M_2>0$ such that

\begin{equation}\label{e:Kkl}\max\{\E |  K_{kl} |^{2q}, \E |  \mathcal{K}_{kl} |^{2q}\} \leq (Cq)^{cq} \mathcal{E}_q \end{equation}

for $E,\eta \in S_{E,\eta}$, $N>M_1$ , $\frac{N\eta}{|\sqrt{\theta}|} > M_2$, $k\neq l\in \{1,...,N\}$, $q \in \mathbb{N}$ with $q\leq c_0 N .$ Assuming $|\,|\J_1| - |\J_2|\,| < Cq$ for some constant $C$, same inequality holds for $K^{(\J_1)}_{(\J_2), kl}, \mathcal{K}^{(\J_1)}_{(\J_2), kl}.$
\end{lemma}

\begin{proof}

The following argument is identical for $K^{(\J_1)}_{(\J_2), kl}, \mathcal{K}^{(\J_1)}_{(\J_2), kl},$ so we work with $K_{kl}$. By the definition of  $K_{kl}$ and using the notation $\epsilon_{k1}, \epsilon_{k_2}$ for $\epsilon_1$ and $\epsilon_2$ as in \eqref{d:epsilonk12} we get that:

\begin{multline} \label{e:Kklbb}
\E | K_{kl} |^{2q} \leq \frac{(C|\sqrt \theta|)^q}{N^{2q}} \left( \E |\epsilon_{k2}|^{2q} + \E \sum_j|\mathcal{G}^{(kl)}_{jj}x_{kj}x_{lj}|^{2q}\right)
\leq
 \frac{(Cq|\sqrt \theta|)^{cq}}{(N\eta)^q}\\
\end{multline}
where $\E |\epsilon_{k2}|$ is bounded using \eqref{e:epsilon2bound}, \eqref{e:Q0bound} and $\E \sum_j|\mathcal{G}^{(kl)}_{jj}x_{kj}x_{lj}|^{2q}$ is bounded by Rosenthal's inequality like $\E|\epsilon_{k1}|^{2q}$ in \eqref{e:Rosenthal}. We also use Lemma \ref{l:Gkk} to bound $\E |\mathcal{G}_{kk}|^{2q}$. Now using \eqref{e:Kklbb}, \eqref{e:form-off-diag}, and Lemma \ref{l:Gkk} we obtain that
\begin{equation}\label{e:Gklbb}
\E|G_{kl}|^{2q} \leq \frac{(Cq)^{cq}}{(N\eta)^q}.
\end{equation}
To improve the bound \eqref{e:upsilonbound1}, we see that using equation \eqref{e:upsilonbound2} and \eqref{e:Gklbb} as well as Lemma \ref{l:Gkk} (also using that $ \frac{2}{(N\eta)^qN^{q/2}} \leq \frac{1}{(N\eta)^{2q}}+ \frac{1}{N^{q}}$), we obtain
\begin{equation}\label{e:upsilonboundoptimal} \E|\Upsilon|^{2q} \leq \left(\frac{Cq}{N\eta}\right)^{cq}\E|\mathrm{Im}\Tr \mathcal{G}|^q   + (Cq)^{cq}\left(\frac{1}{N^q}+\frac{1}{(N\eta)^{2q}}\right)\end{equation}
and using \eqref{e:Gklbb} we can improve the bound on $\E |\epsilon_{2k}|^{2q}$ in \eqref{e:firststep}, which yields \eqref{e:Kkl}.

\end{proof}

\medskip

\begin{lemma} \label{reciprocal}
Assume \eqref{a:4moment} and \eqref{a:truncation} for the entries of $X_N$ as before and let $\theta=E+i\eta \in S_{E,\eta}$.
There exist constants $c,C,M>0$ such that

$$ \E \frac{1}{| \sqrt{\theta} G_{11}^{(\J)}|^{2q } } \leq C^q ,$$

for $\theta\in S_{E,\eta}$, $N\eta > |\sqrt{\theta}| M$, $q\leq c(N \eta)^{1/4} $ and $\J\subset \{1,...,N \}$, with $|\J|\leq 2q.$
\end{lemma}

\begin{proof}
We can take $\J=\emptyset$ as the argument is similar in the general case. We have that:

\begin{multline*} \E \frac{1}{| \sqrt{\theta} G_{11}|^{2q } } = \E |\sqrt{\theta} (1+ (\mathbf{x}^1)^*\mathcal{G}^{(1)}\mathbf{x}^1/N ) |^{2q} \leq C^q + (C|\theta|)^q \E | (\mathbf{x}^1)^*\mathcal{G}^{(1)}\mathbf{x}^1/N |^{2q}
\\
\leq C^q (1+ |\theta|^q\E|(\mathbf{x}^1)^*\mathcal{G}^{(1)}\mathbf{x}^1/N  -\E_{\xx^1} (\mathbf{x}^1)^*\mathcal{G}^{(1)}\mathbf{x}^1 /N |^{2q} +\E | \E_{\xx^1} \sqrt{\theta}(\mathbf{x}^1)^*\mathcal{G}^{(1)}\mathbf{x}^1/N |^{2q} ) .
\end{multline*}

The second term on the RHS is small by Lemma \ref{l:Upsilonbound}.
For the third term, we find that:

\begin{equation}
\E | \E_{\xx^1} \sqrt{\theta}(\mathbf{x}^1)^*\mathcal{G}^{(1)}\mathbf{x}^1/N|^{2q} = \E \left |\frac{1}{N} \sqrt{\theta} \mathrm{Tr} (\mathcal{G}^{(1)}) \right|^{2q} = \E \left | \frac{1}{N} \sqrt{\theta} \left ( \frac{1}{\theta} + \mathrm{Tr} (G^{(1)}) \right) \right |^{2q} \leq C^q ,
\end{equation}
where we used Lemma \ref{l:Gkk} and that $|\Delta_N^{(1)} - \Delta_N|\leq \frac{1}{N\eta} $ as in \eqref{e:minor}.\\
\end{proof}

To estimate the third quantity in \eqref{e:quantities}, we find by \eqref{e:upsilonboundoptimal} that:

\begin{equation}\label{e:oneover}\left| (\mathbb{I}-\E_{\xx^k}) \frac{1}{\sqrt{\theta} G_{kk}}\right| = \left|-\sqrt{\theta} \Upsilon^{(\{k\})} \right|\leq (Cq)^{cq} \mathcal{E}_q.
\end{equation}

Lastly, we also need a bound on $\E \left | \frac{1}{\E_{\xx^1} \frac{1}{\sqrt{\theta} G_{11}}} \right |^q$ which we obtain in the following lemma.

\begin{lemma} \label{lem}
Let $E,\eta \in S_{E,\eta}$, where $\theta =E+i\eta$. There exist constants $c,C,M>0$ such that:

$$ \E \left | \frac{1}{\E_{\xx^1} \frac{1}{\sqrt{\theta} G_{11}}} \right |^q \leq C^q ,$$

for $N \eta \geq |\sqrt{\theta}| M $ and for $q \in \mathbb{N}$ with $q\leq c \left (\frac{N\eta}{|\sqrt{\theta}|} \right )^{1/4} .$
\end{lemma}

\begin{proof}
The proof is similar to Lemma 5.1 in \cite{cacciapuoti2015bounds}.
We define $$\widetilde{G_{11}}=\frac{1}{\E_{\xx^1} \frac{1}{G_{11}}}= - \frac{1}{\theta(1+\Tr\mathcal{G}^{(\{1\})})}. $$

We calculate that $$ \left| \frac{d}{d\eta} \log \widetilde{G_{11}}(E+i\eta) \right|= \left| \frac{d}{d\eta} \log \left(\frac{1}{\theta} \right) + \frac{d}{d\eta} \log \left ( \frac{1}{1+ \Tr\mathcal{G}^{(\{1\})}} \right ) \right|
=
 \left| - \frac{i}{\theta} - \frac{\frac{d}{d\eta}\Tr\mathcal{G}^{(\{1\}) } }{1+\Tr\mathcal{G}^{(\{1\}) }} \right| .$$ \\

We show that $ \left | \frac{d}{d\eta} \Tr\mathcal{G}^{(\{1\}) } \right | \leq \frac{\Im \Tr\mathcal{G}^{(\{1\}) }}{\eta} $ as follows:
\begin{multline}\frac{d}{d\eta} \Tr \mathcal{G}^{(1)} = \sum \limits_{k=1}^N \frac{d}{d\eta} \mathcal{G}^{(\{1\})}_{kk}(\theta) = \sum \limits_{k=1}^N i  ((\mathcal{G}^{(\{1\})})^{2})_{kk}= \sum \limits_{k=1}^N i \langle e_k, (\mathcal{G}^{(\{1\})})^2 e_k \rangle
\\
 \Rightarrow  \left | \frac{d}{d\eta} \mathcal{G}^{(\{1\})} \right | \leq \sum \limits_{k=1}^N || (\mathcal{G}^{(\{1\})})^* e_k|| \ || \mathcal{G}^{(\{1\})} e_k || \leq \sum \limits_{k=1}^N ( (\mathcal{G}^{(\{1\})})^*\mathcal{G}^{(\{1\})} )_{kk} = \sum \limits_{k=1}^N \frac{\Im (\mathcal{G}^{(\{1\})})_{kk}}{\eta} = \frac{\Im \Tr \mathcal{G}^{(\{1\})}}{\eta} .\end{multline}
We conclude that
\begin{equation}  \left | \frac{d}{d\eta} \log \widetilde{G_{11}} \right | \leq \frac{1}{|\theta|} + \frac{\Im \Tr\mathcal{G}^{(\{1\})}}{\eta |1+\Tr\mathcal{G}^{(\{1\})}|} \leq \frac{2}{\eta} ,\end{equation}
yielding that  \begin{equation} \left |\log \widetilde{G_{11}}(E+i\eta) - \log \widetilde{G_{11}}(E+i\eta/s) \right |= \left | \int_{\eta/s}^{\eta} \frac{d}{d\nu} \log \widetilde{G_{11}}(E+i\nu) d\nu \right | \leq \int_{\eta/s}^{\eta} \frac{2}{\nu} d\nu = \log s^2 \end{equation}
and thus $| \widetilde{G_{11}}(E+i\eta) | \leq s^2 | \widetilde{G_{11}}(E+i\eta/s ) | .$
The proof now proceeds with induction on $\eta$ just like in the proof of Lemma \ref{l:Gkk} using the identity \begin{equation}\label{e:tildeGkk}\sqrt{\theta}\widetilde{G_{11}} = \sqrt{\theta} G_{11} + \sqrt{\theta} G_{11} \sqrt{\theta} \widetilde{G_{11}} (\mathbb{I}-\E_{\xx^1} ) (\sqrt \theta G_{11})^{-1} \end{equation}
as well as \eqref{e:oneover} and the results of Lemma \ref{l:Gkk}.
\end{proof}

Lastly, we use the matrix expansion algorithm to take advantage of the fluctuations. Hence the following proposition, analogous to Lemma 4.1 of \cite{cacciapuoti2015bounds}:
\begin{proposition} \label{p:optimalbound}
 Let $\mathcal{E}_q$ be the control parameter as in \eqref{d:controlparam}.
There exist constants $C, M, c_0 >0$ such that

\begin{equation} \label{4.2}
\E \left | \frac{1}{N} \sum \limits_k \sqrt{\theta} \Upsilon^{(\{k\})}  \sqrt{\theta} G_{kk} \right |^{2q} \leq (Cq)^{cq^2} \mathcal{E}_{4q}^{1/2} ,
\end{equation}
for $1 \leq q \leq c_0 \left (\frac{N\eta}{|\sqrt{\theta}|} \right )^{1/8} $ , $\frac{N\eta}{\sqrt{|\theta|}}\geq M $ , $K>0$ , $\theta = E + i \eta \in S_{E, \eta} $.
\end{proposition}

\begin{proof}

To match notation in \cite{cacciapuoti2015bounds}, we introduce $W_k= \sqrt{\theta} \Upsilon_{k} \sqrt{\theta} G_{kk}$ and we split:
\begin{equation}
\label{e:eW2q} \frac{1}{N} \sum \limits_k  W_k = \frac{1}{N} \sum \limits_k (\mathbb{I}-\E_k)W_k + \frac{1}{N}\sum \limits_k \E_k W_k .$$
By H\"older's inequality,
$$ \E \left | \frac{1}{N} \sum \limits_k  W_k \right |^{2q} \leq C^q \E \left | \frac{1}{N} \sum \limits_k  ( \mathbb{I} - \E_k) W_k \right |^{2q} + C^q \E |\E_1 W_1|^{2q} .
\end{equation}
To bound the second term in \eqref{e:eW2q} above, using that $\sqrt{\theta}\Upsilon^{(\{k\})}=-(\mathbb{I}-\E_k)\frac{1}{\sqrt{\theta}G_{kk}}$, we obtain
\begin{equation} \label{4.3}
\E_k W_k = \frac{\E_k [\sqrt{\theta} G_{kk} (\sqrt{\theta}\Upsilon^{(\{k\})})^2]}{\left (\E_k\frac{1}{\sqrt{\theta} G_{kk}}\right )}.
\end{equation}
and applying Lemma \ref{lem} to \eqref{4.3}, we get that:
\begin{align*}
\E | \E_1 W_1 |^{2q} &\leq (\E |\sqrt{\theta} G_{11}|^{8q})^{\frac{1}{4}} \left ( \E \left | \frac{1}{\E_1 \frac{1}{\sqrt{\theta} G_{11}}} \right |^{8q} \right )^{\frac{1}{4}} ( \E | \sqrt{\theta} \Upsilon^{(\{1\})} |^{8q} )^{\frac{1}{2}} \\
&\leq (Cq)^{cq} \left ( \frac{|\theta|^{4q}}{ (N \eta)^{8q} } + \frac{(\Im|\theta| \Delta))^{4q} + \E|\theta \Lambda|^{4q} }{(N\eta)^{4q}}   \right )^{\frac{1}{2}} ,
\end{align*}
which is what we want.

\medskip

In order to handle the first term of \eqref{e:eW2q}, we use the matrix expansion algorithm as in Section 5.2 of \cite{cacciapuoti2015bounds}. We notice that equations (5.7), (5.8), and (5.9) are the basis of the expansion algorithm, and they are equivalent to the following (see e.g. (2.18) in \cite{pillai2014universality}):
\begin{equation}\label{e:algorithmnew}\begin{split}
&\sqrt \theta G_{ij}^{(\T)} = \sqrt \theta G_{ij}^{(\T k)} + \frac{\sqrt \theta G_{ik}^{(\T)}\sqrt \theta G_{kj}^{(\T)}}{\sqrt \theta G_{kk}^{(\T)}} \text{ for } i, j, k \notin \T \text{ and } i,j \neq k,
\\
&\frac{1}{\sqrt \theta G_{ii}^{(\T)}} = \frac{1}{\sqrt \theta G_{ii}^{(\T k)}} - \frac{\sqrt \theta G_{ik}^{(\T)}\sqrt \theta G_{ki}^{(\T)}}{\sqrt \theta G_{ii}^{(\T)}\sqrt \theta G_{ii}^{(\T k)}\sqrt \theta G_{kk}^{(\T)}} \text{ for } i, k \notin \T \text{ and } i \neq k\\
\end{split}\end{equation}
Using the above equation \eqref{e:algorithmnew}, we see that in our case the steps of the expansion algorithm (5.13), (5.14), (5.15) in \cite{cacciapuoti2015bounds} are the same except that each resolvent entry is multiplied by a factor of $\sqrt \theta$. Using our definition of $W$, equation (5.6) in \cite{cacciapuoti2015bounds} becomes analogous to
\begin{equation} \label{formexpand}
( \mathbb{I}-\E_{k_s})W_{k_s} = ( \mathbb{I}-\E_{k_s}) \left [ ( \mathbb{I}-\E_{k_s}) \frac{1}{\sqrt{\theta}G_{k_s k_s}} \right ] \sqrt{\theta} G_{k_s k_s} , s=1,...,2q ,
\end{equation}
so the initial terms of the algorithm are $A^{r}:=\sqrt \theta G_{k_r k_r}$ and $B^{r} :=\frac{1}{\sqrt{\theta}G_{k_r k_r}}$ are the same as (5.16), (5.17) of \cite{cacciapuoti2015bounds} except that each resolvent entry is multiplied by a $\sqrt \theta$.
Then (5.18), (5.19), and (5.20) of \cite{cacciapuoti2015bounds} carry over directly as well as properties (1) through (5) of relevant strings. We then obtain the desired result
\[
\E \left|\frac{1}{N}\sum_k( \mathbb{I}-\E_{k})W_{k}\right|^{2q} \leq (Cq)^{cq^2} \mathcal{E}_{4q}^{1/2}
\]
using the proof of (5.32) of \cite{cacciapuoti2015bounds}. It relies on counting the types of terms that result from the expansion algorithm. Since our algorithm yields the same type and number of terms in each step, the proof in our case will be identical.
In \cite{cacciapuoti2015bounds}, we notice the use of bounds (3.9) and Lemma 5.2 in (5.44) as well as in Case 2, bounds (5.26) and (3.4) in (5.43) and (5.49). We can replace (3.9), Lemma 5.2, (5.26), and (3.4) of \cite{cacciapuoti2015bounds} by our bounds on the relevant quantities in \eqref{e:quantities} as well as our (5.4).

\end{proof}

\begin{proof}[Proof of Theorem \ref{t:sti}]
 By Proposition \ref{b-R}, in order to control $\Lambda$, we need to control high moments of $R=N^{-1} \sum \limits_{k=1}^N G_{kk}(T_k + \Upsilon^{(\{k\})})$. Taking expectation of $2q$  power we obtain
\begin{equation}\label{e:Ypsilonpower}\E| \theta R|^{2q} \leq  C^q \left ( \E \left | \frac{1}{N} \sum \limits_k \sqrt{\theta} T_k \sqrt{\theta} G_{kk} \right |^{2q} + \E \left | \frac{1}{N} \sum \limits_k \sqrt{\theta} \Upsilon^{(\{k\})} \sqrt{\theta} G_{kk} \right |^{2q} \right )  .\end{equation}

For the first term by \eqref{e:minor}, we obtain
\begin{align} \label{4.1}
\E \left| \frac{1}{N} \sum \limits_{k} \sqrt{\theta}T_k \sqrt{\theta}G_{kk} \right|^{2q}\leq C^q \frac{1}{N^{2q} |\theta|^q}   .
\end{align}
while the second term is handled in Proposition \ref{p:optimalbound}, yielding that
$$ \E|\theta R|^{2q} \leq (Cq)^{cq^2} \mathcal{E}_{4q}^{1/2}. $$

Here we are able simplify the analysis in \cite{cacciapuoti2015bounds} by only using the bounds proportional to $R$ from Proposition \ref{b-R} to control $E|\Lambda|^{2q}$ on $S_{E, \eta}$ and $E|\Im\Lambda|^{2q}$. Our simplifications carry over also to the Wigner case. We can assume that

$$[\Im(|\theta| \Delta)]^{2q} + \E|\theta \Lambda|^{2q} \geq \frac{|\theta|^{2q}}{(N\eta)^{2q}} ,$$

(otherwise $\E | \Lambda|^{2q} \leq \frac{1}{(N\eta)^{2q}} ,$ as we want)
and in this case:

$$ \mathcal{E}_{2q} = \frac{1}{N^{2q} |\theta|^{q}} + \frac{\Im (|\theta| \Delta )]^{2q} + \E |\theta \Lambda|^{2q}}{(N\eta)^{2q}} \leq \frac{\eta^q + \Im (|\theta| \Delta )]^{2q} + \E |   \theta \Lambda|^{2q}}{(N\eta)^{2q}} ,$$ \\
Using the bound proportional to $|R|$ from Proposition \ref{b-R}, we obtain

\begin{align*}
\E |\theta \Lambda|^q &\leq \frac{C^q \E| \theta R|^q}{|\Delta+\frac{1}{2}|^q} \leq \frac{(Cq)^{cq^2}}{|\Delta+\frac{1}{2}|^{q}} \left ( \frac{\eta^q + [\Im(|\theta|\Delta)]^{2q}}{(N\eta)^{2q}} \right)^{1/2} = \frac{(Cq)^{cq^2}}{|\Delta+\frac{1}{2}|^q} \frac{|\theta|^q}{(N\eta)^q} \left ( \frac{\eta^q}{|\theta|^{2q}} + [\Im (\Delta)]^{2q} \right )^{1/2} \\
&\leq \frac{(Cq)^{cq^2} |\theta|^q } {(N\eta)^q} \left [ \left (\frac{\sqrt{\eta}}{|\theta||\Delta+\frac{1}{2}|} \right )^q+ \left ( \frac{\Im\Delta}{|\Delta+\frac{1}{2}|} \right )^q \right ].
\end{align*} \\
To obtain the desired bound we now note that
$\Im\Delta \leq |\Delta+\frac{1}{2}|$ and $\frac{\sqrt{\eta}}{|\theta| |\Delta+\frac{1}{2}|} \leq C$ on our domain. The first one follows easily and for the second one we argue as follows:
$$ \frac{\sqrt{\eta}}{|\theta||\Delta+\frac{1}{2}|} = \frac{2\sqrt{\eta}}{\sqrt{|\theta|}\sqrt{|\theta-4|}} ,$$
and by triangle inequality either $|\theta|\geq 2$ or $|\theta-4|\geq 2$. Then in the first case, we use the bound $\sqrt{\eta} \leq \sqrt{|\theta-4|} $ and in the second case the bound $\sqrt{\eta} \leq \sqrt{|\theta|} $. \\
Overall, this implies that
\begin{equation}\label{e:finalLambdabound}\mathbb{P} \left( |\Delta_N - \Delta|\geq \frac{K}{ N\eta} \right)\leq \frac{(N\eta)^q}{K^q}\E | \Lambda|^q \leq \frac{(Cq)^{cq^2}}{K^q} ,
\end{equation}
for $1 \leq q \leq c_0 \left (\frac{N\eta}{|\sqrt{\theta}|} \right )^{1/8} $ , $\frac{N\eta}{\sqrt{|\theta|}}\geq M $ , $K>0$ , $\theta = E + i \eta \in S_{E, \eta} $ .

\end{proof}

\section{Convergence of the counting function}\label{s:counting}
In this section we prove Theorem \ref{t:counting}.
\begin{proof}[Proof of Theorem \ref{t:counting}]
Let $0<E \leq 4$. We will use a Pleijel argument from \cite{pleijel1963theorem}, recently used in obtaining estimates on a measure from estimates on a Stieltjes transform in \cite{erdHos2016fluctuations}. We start from the following equations (equations (13) and (14) in  \cite{erdHos2016fluctuations}, following from equation (5) of \cite{pleijel1963theorem}):
\begin{equation}\label{e:pleijel}
\mu(-K, E) = \frac{1}{2\pi i}\int_{L(z_0)}m_{\mu}(z)dz + \frac{\eta_0}{\pi}\Re m_\mu(z_0) + O(\eta_0 \Im m_\mu(z_0))
\end{equation}
and
\begin{equation}\label{e:pleijel2}
\mu(x, x') = \frac{1}{2\pi i}\int_{\gamma(x, x')}m_{\mu}(z)dz + O(\eta_0(| m_\mu(x + i\eta_0)| + m_\mu(x' + i\eta_0)|)
\end{equation}
where $m_\mu$ is the Stieltjes transform of $\mu$ and $L(z_0)$ is a contour as in Figure \ref{f:domains} (see also \cite{erdHos2016fluctuations} Fig 1A), namely connects with line segments the points $E-i\eta_0, E-iQ, -1-iQ , -1+ iQ, E+ iQ, E+ i\eta_0$ in that order with an arbitrarily chosen constants $-1$ and $Q$, and $\gamma(x, x')$ is the contour connecting $x+i\eta_0, x+iQ, x'+iQ,$ and $x'+i\eta_0$ in that order.

\medskip
By Markov inequality we obtain that
\begin{equation}
\PP\left(|n_N(E) - n_{MP}(E)| \geq \frac{K \log N}{N}\right) \leq \frac{N^q\E(|n_N(E) - n_{MP}(E)|^q)}{\left(K \log N\right)^q}
\end{equation}
Then using \eqref{e:pleijel} and taking $z_0 := E + i\eta_0$ with $\eta_0 := \frac{M\sqrt E}{N}$ with $M$ as in Theorem \ref{t:sti} we obtain that
\begin{multline}
 \E(|n_N(E) - n_{MP}(E)|^q) = \E\left|\frac{1}{2\pi i}\int_{L(z_0)}\Lambda(z)dz + \frac{\eta_0}{\pi} \Re\Lambda(z_0) + O\left(\eta_0 (\Im \Delta_N(z_0)+ \Im \Delta_{MP}(z_0)\right)\right|^q
\\
\leq  C^q\left(\E\left|\int_{L(z_0)}\Lambda(z) dz\right|^q + O\left(\eta_0^q \E|\Lambda(z_0)|^q+\eta_0^q \Im \Delta_{MP}(z_0)^q\right)\right),
\end{multline}
noting that the constant in the $O$ comes from the Pleijel formula and is uniform in the matrix randomness. We study the above expression one term at a time.
For $E \leq 4$ we can bound the second term as follows
\begin{equation}\label{e:etalambda}
 \eta_0^q \E|\Lambda(z_0)|^q \leq  \eta_0^q \frac{Cq^{q^2}}{(N\eta_0 )^q} \leq \frac{Cq^{q^2}}{N^q}.
\end{equation}
The third term is bounded using the above inequality \eqref{e:etalambda} on $\Lambda$ as well as
\begin{equation}\label{e:etaoverE}
 \eta_0 \Im \Delta_{MP} \leq \frac{C\eta_0}{\sqrt E} \leq \frac{CM}{N}.\end{equation}

Now for the integral, we note that it suffices to study the part of the contour where $\Im z > 0$ since $\Lambda(\bar z) = \overline{\Lambda( z)}$. Thus we obtain
\begin{equation}
\E\left|\int_{L(z_0)}\Lambda(z) dz\right|^q \leq C^q \left(\E \left|\int_0^{\eta_0} \Lambda(-1 + iy) dy \right|^q +\E \left|\int_{\eta_0}^Q \Lambda(-1 + iy) - \Lambda(E + iy)dy \right|^q+  \E \left|\int_{-1}^E  \Lambda(x + iQ) dx \right|^q\right)
\end{equation}
Since all eigenvalues are positive we bound $\Lambda$ for $-1<0$ by $\Lambda(-1+i\eta) < 2$ which yields
\begin{equation}
\left|\int_0^{\eta_0} \Lambda(-1 + iy) dy \right|^q \leq \left(\int_0^{\eta_0} |\Lambda(-1 + iy)| dy \right)^q \leq C^q \eta_0^q.
\end{equation}
 Next we note that
\begin{equation}
\E \left|\int_{-1}^E  \Lambda(x + iQ) dx \right|^q \leq \frac{(Cq)^{q^2}}{(NQ)^q}
\end{equation}

Now we can bound the expected value of the integrals $ \E \left ( \int_{\eta_0}^Q | \Lambda( E+iy) |dy \right )^q $ and $ \E \left ( \int_{\eta_0}^Q | \Lambda( -1 +iy) |dy \right )^q $ for $E \leq 4$, noting that the argument is identical at $E$ and $-1$,
\begin{align*}
&\E \left ( \int_{\eta_0}^Q | \Lambda(E +iy) |dy \right )^q = \E \int_{\eta_0}^Q | \Lambda(E +iy_1) |dy_1 \int_{\eta_0}^Q | \Lambda(E+iy_2) |dy_2 \cdots \int_{\eta_0}^Q | \Lambda(E+iy_q) |dy_q \\
&= \E \int_{\eta_0}^Q \cdots \int_{\eta_0}^Q \prod_{j=1}^q | \Lambda(E +iy_j) | \prod_{j=1}^q dy_j = \int_{\eta_0}^Q \cdots \int_{\eta_0}^Q \E \prod_{j=1}^q | \Lambda(E +iy_j) | \prod_{j=1}^q dy_j \\
&\leq
 \int_{\eta_0}^Q \cdots \int_{\eta_0}^Q \prod_{j=1}^q \left (\E | \Lambda(E+iy_j) |^q \right )^{\frac{1}{q}}  \prod_{j=1}^q dy_j
\leq
 \frac{1}{N^q}  \int_{\eta_0}^Q \cdots \int_{\eta_0}^Q \prod_{j=1}^q \frac{(Cq)^{cq}}{y_j} \prod_{j=1}^q dy_j \\
&= \frac{(Cq)^{cq^2}}{N^q} \left ( \int_{\eta_0}^Q \frac{1}{y} dy \right )^q \leq (Cq)^{cq^2}  \frac{(\log N)^q}{N^q}
\end{align*}
where we can apply \eqref{e:finalLambdabound} inside the integral because our estimates on $\Lambda$ are uniform on compact sets.
\medskip

To prove the second part of \eqref{e:counting}, we use the  \eqref{e:pleijel2} and study the interval $[-E, E]$, noting that $
n_N(E) = \mathcal{N}([-E, E])/N
$
 and $n_{MP}(E) = n_{MP}(E) - n_{MP}(-E)$.
The corresponding integral can be bounded similar to above
\begin{multline}
\E \left | \int_{-E}^E  \Lambda(x+i\eta_0) - \Lambda(x-i\eta_0) dx \right |^q = \E \left | \int_{-E}^E  2\Im \Lambda(x+i\eta_0)\right |^q
\\
= \int_{-E}^E \cdots \int_{-E}^E \E\prod_{j = 1}^{q}  |2\Im \Lambda(x_j + i\eta_0) | d x_1 \cdots d x_q \leq \frac{(Cq)^{cq^2}E^q}{(N\eta_0)^q} \leq    \frac{(Cq)^{cq^2}(\sqrt E)^q}{M^q}
\end{multline}
and, similar to \eqref{e:etaoverE}
\begin{equation}\label{e:deltabound}
\max\{\eta_0 \Delta_{MP}(-E), \eta_0 \Delta_{MP}(E)\} \leq \frac{\eta_0}{\sqrt{E}}\leq \frac{M}{N}
\end{equation}
which together with \eqref{e:etalambda} yields the second part of \eqref{e:counting} for $E<4$.

To establish the \eqref{e:counting} for $E > 4$, we use \eqref{e:counting} for $E = 4$ to establish bounds on the number of eigenvalues outside the spectrum. Letting $\mathcal{N}_I$ be the number of eigenvalues in an interval $I$, we see that
\begin{equation}
 \mathcal{N}_{(4, \infty)} = N - N n(4) = N(n_{MP}(4) - n(4))
\end{equation}
which by \eqref{e:counting} for $E = 4$ yields that
\begin{equation}\label{e:outside}
\PP\left(\frac{\mathcal{N}_{(4, \infty)}}{N} > \frac{K\log N}{N}\right) \leq \frac{(Cq)^{q^2}}{K^q}
\end{equation}
and for $E > 4$,
\begin{equation}
\PP\left(|n_N(E) - n_{MP}(E)| \geq \frac{K \log N}{N}\right) \leq \PP\left(\frac{\mathcal{N}_{(4, \infty)}}{N} > \frac{K\log N}{N}\right)
\end{equation}
thus \eqref{e:outside} gives the desired bound.
\end{proof}

\section{Rigidity of the eigenvalues}

The aim of this section is a proof of Theorem \ref{t:rigidity}.
\begin{proof}[Proof of Theorem \ref{t:rigidity}]
Let $\alpha \leq \frac{N}{2}$. We will make use of the following inequalities near the hard edge and away from the soft edge:
$$ c \sqrt{x} \leq  n_{MP}(x) \leq C \sqrt{x} , $$ and $$ c  n_{MP}(x)^{-1} \leq \rho(x) \leq C  n_{MP}(x)^{-1} .$$
valid for $x\in (0,3] .$ The second inequality implies that
\beq \label{14}
c \frac{N}{a} \leq \rho(\gamma_a) \leq C  \frac{N}{a}
\eeq
for any $a \leq \frac{N}{2}.$

For $\varepsilon > 0$, we have that
\begin{align*}
\mathbb{P} \left ( | \lambda_a - \gamma_a | \geq K \epsilon \frac{a}{N}  \right ) \\
&\leq \mathbb{P} \left ( | \lambda_a - \gamma_a | \geq K\varepsilon\frac{a}{N}   \ \text{and} \ \lambda_a \leq \gamma_a \right ) + \mathbb{P} \left ( | \lambda_a - \gamma_a | \geq K\varepsilon \frac{a}{N}   \ \text{and} \ \lambda_a >\gamma_a \right ) \\
&= A + B \ .
\end{align*}

We consider first the term $A.$ We set
$$\ell = K \varepsilon\frac{a}{N}.$$
From $ \lambda_a \leq \gamma_a $ and $| \lambda_a - \gamma_a | \geq \ell$ we find that
$\lambda_a \leq \gamma_a - \ell .$ This implies that $ n_N( \gamma_a - \ell ) \geq \frac{a}{N} = n_{MP}(\gamma_a)$. By the mean value theorem for the function $n_{MP}$, there exists a point $x^* \in [\gamma_a - \ell , \gamma_a ] $ such that
$n_{MP}(\gamma_a ) - n_{MP}(\gamma_a - \ell ) = \rho(x^*) \ell$, yielding that
\begin{multline}n_N(\gamma_a - \ell) - n_{MP}(\gamma_a - \ell) = n_N(\gamma_a-\ell) - n_{MP}(\gamma_a) + \rho(x^*)\ell
\geq \rho(x^*) \ell \geq \rho(\gamma_a)K \varepsilon \frac{a}{N} \geq  cK \varepsilon ,
\end{multline}
because $\rho$ is non-increasing, $a < N/2$, and from \eqref{14}.
Setting $\varepsilon = \frac{\log N}{N} $ we deduce from Theorem \ref{t:counting} that
\begin{equation} A \leq \PP \left ( |n_N(\gamma_a - \ell) - n_{MP}(\gamma_a - \ell) | \geq  \frac{cK\log N}{N} \right ) \leq \frac{(Cq)^{cq^2}}{K^q}
\end{equation}
For $a \leq \log N$, set $\varepsilon = \frac{a}{N} \geq c\sqrt{\gamma_a}$ to obtain
\begin{equation}
A \leq \PP \left ( |n_N(\gamma_a - \ell) - n_{MP}(\gamma_a - \ell) | \geq cK \sqrt{(\gamma_a - \ell)_+} \right )
\leq \frac{(Cq)^{cq^2}}{K^q}.
\end{equation}
\\

We now estimate the term $B.$ From the estimate $n_{MP}(x) \sim \sqrt{x}$ near the hard edge, we have that
$$ \gamma_a \leq C \left ( \frac{a}{N} \right )^2 ,$$
for some constant $C>0$ for all $a < N/2$. We consider the number
$$
 y = 2C \left ( \frac{a}{N} \right )^2
 $$
and we further consider the cases that $\gamma_a + \ell \leq y $ or $ \gamma_a + \ell > y .$ \\

In the first case since $\lambda_a > \gamma_a $ and $| \lambda_a - \gamma_a | \geq \ell ,$ we have that  $\lambda_a > \gamma_a + \ell $ and so $ n_N( \gamma_a + \ell) \leq \frac{a}{N} = n_{MP}(\gamma_a).$

Hence, from the mean value theorem, we find $x^* \in [\gamma_a , \gamma_a + \ell ] \subset [\gamma_a , y ] $ such that $n_{MP}(\gamma_a + \ell) - n_{MP}(\gamma_a ) = \rho(x^*) \ell$, yielding that
\begin{equation*}
n_{MP}(\gamma_a + \ell) - n_N(\gamma_a + \ell) = n_{MP}(\gamma_a) - n_N(\gamma_a+\ell) + \rho(x^*)\ell
\geq \rho(x^*) \ell = \rho(x^*)  K\varepsilon \frac{a}{N} \geq \rho(y) K\varepsilon \frac{a}{N} \geq  c K\varepsilon,
\end{equation*}
where we used that $\rho$ is nonincreasing and that $\rho(y) \geq \frac{c}{\sqrt{y}}$ near the hard edge.
Setting $\varepsilon = \frac{\log N}{N}$ and using Theorem \ref{t:counting}, we conclude that
\begin{equation} B \leq \PP \left( |n_{MP}(\gamma_a + \ell) - n_N(\gamma_a + \ell) | \geq cK \frac{ \log N}{N}\right)\leq \frac{(Cq)^{cq^2}}{K^q},
\end{equation}
as required. For rigidity at the hard edge equation \eqref{e:rigidityedge}, let $\varepsilon = \frac{a}{N} $ to obtain
\begin{multline}
B\leq \PP \left( |n_{MP}(\gamma_a + \ell) - n_N(\gamma_a + \ell) | \geq \frac{cKa}{N} \right) \\
\leq \PP \left( |n_{MP}(\gamma_a + \ell) - n_N(\gamma_a + \ell) | \geq c\sqrt{K} \sqrt{\gamma_a + \ell}  \right)
\leq \frac{(Cq)^{cq^2}}{K^{q/2}} , \\ \end{multline}
where the second line follows as before because $\sqrt{\gamma_a} \leq c\frac{a}{N}$ and $\ell = \frac{Ka^2}{N^2}$.
\\

In the other case we have that $ \gamma_a + \ell > y$ so the inequality $\lambda_a > \gamma_a + \ell$ implies that $\lambda_a > y$ and therefore $n_N (y) \leq \frac{a}{N} = n_{MP}(\gamma_a).$ Hence from the mean value theorem there exists $x^* \in [\gamma_a, y]$ such that $n_{MP}(y) - n_{MP}(\gamma_a ) = \rho(x^*) \ell$, which yields
\begin{equation*}
n_{MP}(y) - n_N(y) = n_{MP}(\gamma_a) - n_N(y) + \rho(x^*)\ell \geq \rho(x^*) \ell = \rho(x^*)  K \varepsilon \frac{a}{N} \geq \rho(y) K\varepsilon \frac{a}{N} \geq  c K\varepsilon,
\end{equation*}
and we can conclude \eqref{e:rigiditybulk} and \eqref{e:rigidityedge} as above. This finishes the proof of Theorem \ref{t:rigidity}.
\end{proof}

\section*{Appendix}
Here we state Rosenthal's and Burkholder's inequalities adapted to complex variables and non-Hermitian bilinear forms (useful for $\mathcal G$). Given $x_1, ..., x_N$ with i.i.d. real and imaginary parts with $\E \Re x_j = \E \Im x_j = 0$ and $\E |\Re x_j|^2 = \E |\Im x_j|^2 = 1/2$, as in our setup. We also assume that $\E |x_j|^p \leq \mu_{p, N}$ for $p\geq 1$, so the moments exist but may depend on $N$.
\medskip

The following lemma is our version of Rosenthal's inequality, which is easy to prove by separating real and imaginary parts of the random variables and using Lemma 7.1 of \cite{gotze2016optimal}:
\begin{lemma}[Rosenthal's inequality] \label{l:rosenthal} There exists a constant $C_1$ such that
\begin{equation}
\E|\sum_{j=1}^N a_j x_j|^{p} \leq C_1^p p^p \left(\left(\sum_{j=1}^N |a_j|^2\right)^{p/2}+ \mu_{p, N}\sum_{j=1}^N|a_j|^p\right).
\end{equation}
\end{lemma}

The following lemma is our version of Burkholder's inequality. This is an extension of Lemma 7.3 of \cite{gotze2016optimal} for complex entries.
\begin{lemma}[Burkholder's Inequality]\label{l:burkholder}
Let $Q = \sum_{1 \leq j \leq N, 1 \leq k \leq N, j \neq k}a_{jk}x_j \overline{x_k}.$ Then there exists absolute constants $C_1, C_2$ such that
\begin{align*}
\E |Q|^q &\leq (C_1q)^q \left ( \E \left [ \sum \limits_{j=2}^n \left | \sum \limits_{k=1}^{j-1} a_{jk} x_k \right |^2 \right ]^{q/2} + \mu_{q,N} \sum \limits_{j=2}^n \E \left | \sum \limits_{k=1}^{j-1} a_{jk} x_k \right |^q \right ) \\
&+ (C_2q)^q \left ( \E \left [ \sum \limits_{j=2}^n \left | \sum \limits_{k=1}^{j-1} a_{kj} x_k \right |^2 \right ]^{q/2} + \mu_{q,N} \sum \limits_{j=2}^n \E \left | \sum \limits_{k=1}^{j-1} a_{kj} x_k \right |^q \right ) .
\end{align*}
\end{lemma}
\begin{proof}
We introduce the random variables
\beq
\xi_j := x_j \sum_{k=1}^{j-1} a_{jk} x_k \ , \ \widehat{\xi_j} := x_j \sum_{k=1}^{j-1} a_{kj} x_k
\eeq
We let $\mathcal{R}_j := \sigma( \xi_1,...,\xi_j )$ be the sigma-algebra generated by the first $j$ random variables $\xi_1,...,\xi_j$. We observe that the $\xi_j$ and $\widehat{\xi}_j $ are $\mathcal{R}_j-$measurable with $\E [ \xi_j | \mathcal{R}_{j-1} ] = 0$ and $\E [ \widehat{\xi_j} | \mathcal{R}_{j-1} ] = 0,$ which means that they form martingale differences. Next, we write $Q$ as
\beq
Q = \sum \limits_{j=2}^{n} \xi_j + \sum \limits_{j=2}^n \widehat{\xi_j}
\eeq
and so,
$$ \E |Q|^q \leq C^q \E \left | \sum \limits_{j=2}^n \xi_j \right |^q + C^q \E \left | \sum \limits_{j=2}^n \widehat{\xi_j} \right |^q ,$$
We now apply a general Burkholder-Rosenthal Inequality see e.g. \cite{osekowski2012note}, analogous to Lemma 7.2 from \cite{gotze2016optimal},  to the martingale difference sequences $\xi_1,...,\xi_n$ and $\widehat{\xi_1},..., \widehat{\xi_n}.$ We will evaluate $\E[ \xi_j^2 | \mathcal{R}_{j-1} ]$ and $\E |\xi_j |^q ,$ for $j=1,...,n.$ The case is similar for the $\widehat{\xi_1},..., \widehat{\xi_n} $ random variables. We therefore observe that:
\begin{align*}
\E[ |\xi_j|^2 | \mathcal{R}_{j-1} ] &= \E|x_j|^2  \left | \sum \limits_{k=1}^{j-1} a_{jk} x_k \right |^2 = \left | \sum \limits_{k=1}^{j-1} a_{jk} x_k \right |^2 \\
\E | \xi_j |^q &= \E | \zeta_j|^q \E \left | \sum \limits_{k=1}^{j-1} a_{jk} x_k \right |^q \leq \mu_{q,N} \ \E \left | \sum \limits_{k=1}^{j-1} a_{jk} x_k \right |^q,
\end{align*}
and the lemma follows.
\end{proof}

\bibliographystyle{plain}
\bibliography{Local_covariance}

\def\cprime{$'$}
\begin{thebibliography}{10}

\bibitem{cacciapuoti2013local}
Claudio Cacciapuoti, Anna Maltsev, and Benjamin Schlein.
\newblock Local {M}archenko-{P}astur law at the hard edge of sample covariance
  matrices.
\newblock {\em Journal of Mathematical Physics}, 54(4):043302, 2013.

\bibitem{cacciapuoti2015bounds}
Claudio Cacciapuoti, Anna Maltsev, and Benjamin Schlein.
\newblock Bounds for the {S}tieltjes transform and the density of states of
  wigner matrices.
\newblock {\em Probability Theory and Related Fields}, 163(1-2):1--59, 2015.

\bibitem{erdos2012local}
L{\'a}szl{\'o} Erd{\H{o}}s, Benjamin Schlein, Horng-Tzer Yau, and Jun Yin.
\newblock The local relaxation flow approach to universality of the local
  statistics for random matrices.
\newblock In {\em Annales de l'IHP Probabilit{\'e}s et statistiques},
  volume~48, pages 1--46, 2012.

\bibitem{erdHos2016fluctuations}
L{\'a}szl{\'o} Erd{\H{o}}s, Dominik Schr{\"o}der, et~al.
\newblock Fluctuations of functions of {W}igner matrices.
\newblock {\em Electronic Communications in Probability}, 21, 2016.

\bibitem{gotze2016optimal}
Friedrich G{\"o}tze and Alexander Tikhomirov.
\newblock Optimal bounds for convergence of expected spectral distributions to
  the semi-circular law.
\newblock {\em Probability Theory and Related Fields}, 165(1-2):163--233, 2016.

\bibitem{gotze2014rate}
Friedrich G{\"o}tze and AN~Tikhomirov.
\newblock Rate of convergence of the expected spectral distribution function to
  the {M}archenko--{P}astur law.
\newblock {\em arXiv preprint arXiv:1412.6284}, 2014.

\bibitem{gustavsson2005gaussian}
Jonas Gustavsson.
\newblock Gaussian fluctuations of eigenvalues in the {GUE}.
\newblock In {\em Annales de l'IHP Probabilit{\'e}s et statistiques},
  volume~41, pages 151--178, 2005.

\bibitem{landon2021single}
Benjamin Landon, Patrick Lopatto, and Philippe Sosoe.
\newblock Single eigenvalue fluctuations of general {W}igner-type matrices.
\newblock {\em arXiv preprint arXiv:2105.01178}, 2021.

\bibitem{landon2020applications}
Benjamin Landon, Philippe Sosoe, et~al.
\newblock Applications of mesoscopic {CLT}s in random matrix theory.
\newblock {\em Annals of Applied Probability}, 30(6):2769--2795, 2020.

\bibitem{marchenko1967distribution}
Vladimir~Alexandrovich Marchenko and Leonid~Andreevich Pastur.
\newblock Distribution of eigenvalues for some sets of random matrices.
\newblock {\em Matematicheskii Sbornik}, 114(4):507--536, 1967.

\bibitem{osekowski2012note}
Adam Osekowski.
\newblock A note on {B}urkholder-{R}osenthal inequality.
\newblock {\em Bull. Pol. Acad. Sci. Math}, 60(2):177--185, 2012.

\bibitem{pillai2014universality}
Natesh~S Pillai, Jun Yin, et~al.
\newblock Universality of covariance matrices.
\newblock {\em The Annals of Applied Probability}, 24(3):935--1001, 2014.

\bibitem{pleijel1963theorem}
{\AA}ke Pleijel.
\newblock On a theorem by {P. M}alliavin.
\newblock {\em Israel Journal of Mathematics}, 1(3):166--168, 1963.

\bibitem{su2006gaussian}
Zhonggen Su.
\newblock Gaussian fluctuations in complex sample covariance matrices.
\newblock {\em Electronic Journal of Probability}, 11:1284--1320, 2006.

\bibitem{tao2010random}
Terence Tao and Van Vu.
\newblock Random matrices: The distribution of the smallest singular values.
\newblock {\em Geometric And Functional Analysis}, 20(1):260--297, 2010.

\bibitem{tao2012random}
Terence Tao, Van Vu, et~al.
\newblock Random covariance matrices: {U}niversality of local statistics of
  eigenvalues.
\newblock {\em The Annals of Probability}, 40(3):1285--1315, 2012.

\bibitem{wang2012random}
Ke~Wang.
\newblock Random covariance matrices: {U}niversality of local statistics of
  eigenvalues up to the edge.
\newblock {\em Random Matrices: Theory and Applications}, 1(01):1150005, 2012.

\end{thebibliography}

\end{document}